\documentclass[11pt,reqno]{amsart}
\usepackage{a4}

\usepackage{enumerate}
\usepackage{exscale}
\usepackage{amssymb,amsmath}
\newcommand{\bea}{\begin{eqnarray*}}
\newcommand{\eea}{\end{eqnarray*}}
\newcommand{\beq}{\begin{eqnarray}}
\newcommand{\eeq}{\end{eqnarray}}
\newcommand{\RR}{\mathbb{R}}
\newcommand{\NN}{\mathbb{N}}
\newcommand{\CC}{\mathbb{C}}
\newcommand{\ZZ}{\mathbb{Z}}
\newcommand{\TT}{\mathbb{T}}
\newcommand{\EE}{\mathbb{E}}
\newcommand{\PP}{\mathbb{P}}

\newcommand{\cA}{\mathcal{A}}
\newcommand{\cB}{\mathcal{B}}

\newcommand{\cF}{\mathcal{F}}

\newcommand{\cJ}{\mathcal{J}}

\newcommand{\ddiv}{\mathop{\mathrm{div}}\nolimits}
\newcommand{\grad}{\mathop{\mathrm{grad}}\nolimits}

\newcommand{\wt}{\widetilde}

\newcommand{\bigtimes}{\mathop{\mbox{\Large$\times$}\!\!\!\!\!\!\!\phantom{\bigotimes}}\limits}
\renewcommand{\r}{\right}
\renewcommand{\l}{\left}
\newcommand{\la}{\l \langle}
\newcommand{\ra}{\r \rangle}
\newcounter{smalllist}

\usepackage{bbm} 
\newcommand{\1}{\mathbbm 1} 

\newcommand{\rel}{{\mathrm{rel}}}
\newcommand{\per}{{\mathrm{per}}}

\newcommand{\sob}{{\mathrm{sob}}}
\newcommand{\eff}{{\mathrm{eff}}}

\newcommand{\Dir}{{\mathrm{D}}}

\newcommand{\dom}{\mathop{\mathrm{dom}}}

\newcommand{\injrad}{\mathop{\mathrm{inj\, rad}}}
\newcommand{\bd}{\partial}
\renewcommand{\phi}{\varphi}
\renewcommand{\epsilon}{\varepsilon}  
\newcommand{\eps}{\epsilon}
\newcommand{\Cci}[1]{C_{\mathrm c}^\infty ({#1})} 
\newcommand{\Cb}[1]{C_{\mathrm b} ({#1})} 
\newcommand{\dd}{\,\mathrm{d}}
\newcommand{\und}{\quad\text{and}\quad}
\newcommand{\vol}{\mathrm{vol}}
\newcommand{\supp}{\mathop{\mathrm{supp}}}
\newcommand{\Id}{{\mathop{\mathrm{Id}}}}
\newcommand{\Tr}{\mathop{\mathrm{Tr}}\nolimits}

\DeclareMathOperator{\meas}{\mathrm{meas}}

\newtheorem{thm}{Theorem}[section]

\newtheorem{prp}[thm]{Proposition}
\newtheorem{lem}[thm]{Lemma}
\newtheorem{cor}[thm]{Corollary}

\theoremstyle{definition}
\newtheorem{dfn}[thm]{Definition}

\newtheorem{asp}[thm]{Assumption}
\theoremstyle{remark}
\newtheorem{rem}[thm]{Remark}

\newtheorem{exm}[thm]{Example}

\newcommand{\Hm}[1]{\leavevmode{\marginpar{\tiny%
$\hbox to 0mm{\hspace*{-0.5mm}$\leftarrow$\hss}%
\vcenter{\vrule depth 0.1mm height 0.1mm width \the\marginparwidth}%
\hbox to 0mm{\hss$\rightarrow$\hspace*{-0.5mm}}$\\\relax\raggedright #1}}}

\begin{document}
\title[Integrated density of states for random metrics]{Continuity
  properties of the integrated density of states on manifolds}
\author[D.~Lenz]{Daniel Lenz}
\author[N.~Peyerimhoff]{Norbert Peyerimhoff}
\author[O.~Post]{Olaf Post}
\author[I.~Veseli\'c]{Ivan Veseli\'c}

\address[D.~Lenz]{Fakult\"at f\"ur Mathematik, TU Chemnitz,
                   D-09107 Chemnitz, Germany}
\email{dlenz@mathematik.tu-chemnitz.de}
\urladdr{www.tu-chemnitz.de/mathematik/mathematische\_physik/}

\address[N.~Peyerimhoff]{Department of Mathematical Sciences,
  Durham University, Science Laboratories South Road, Durham, DH1
  3LE, Great Britain}
\email{norbert.peyerimhoff@durham.ac.uk}
\urladdr{www.maths.dur.ac.uk/~dma0np/}

\address[O.~Post]{Institut f\"ur Mathematik,
         Humboldt-Universit\"at zu Berlin,
         Rudower Chaussee~25,
         12489 Berlin,
         Germany}
\email{post@math.hu-berlin.de}
\urladdr{www.math.hu-berlin.de/\protect{\char126}post/}

\address[I.~Veseli\'c]{Fakult\"at f\"ur Mathematik,TU Chemnitz,
  D-09107 Chemnitz, Germany \\ \hspace*{1,3em}\& Emmy-Noether Programme of the DFG}
\urladdr{www.tu-chemnitz.de/mathematik/schroedinger/members.php}

\date{\jobname.tex, \today}

\keywords{integrated density of states, periodic and random operators,
  Schr\"odinger operators on manifolds, continuity properties}
\subjclass[2000]{35J10; 82B44}


\begin{abstract}
  We first analyze the integrated density of states (IDS) of periodic
  Schr\"odinger operators on an amenable covering manifold. A
  criterion for the continuity of the IDS at a prescribed energy is
  given along with examples of operators with both continuous and
  discontinuous IDS'.

  Subsequently, alloy-type perturbations of the periodic operator are
  considered. The randomness may enter both via the potential and the
  metric. A Wegner estimate is proven which implies the continuity of
  the corresponding IDS. This gives an example of a discontinuous
  "periodic" IDS which is regularized by a random perturbation.
\end{abstract}
\maketitle


\section{Introduction}
This paper is devoted to the study of continuity properties of the
integrated density of states (IDS) of ergodic Schr\"odinger operators
on manifolds.  The IDS is a distribution function introduced in the
quantum theory of solids which measures the number of electron levels
per unit volume up to a given energy.  It allows to calculate all
basic thermodynamic properties of the corresponding non-interacting
electron gas, like e.g.~the free energy.

This article is concerned with the H\"older continuity of the IDS for
particular random Schr\"odinger operators on manifolds. The continuity
of the IDS is a matter of interest both for physicists
(e.g.~\cite{Wegner-81}) and geometers (e.g.~\cite{DodziukLMSY-03}). It
has been intensely studied in the theory of localization for random
Schr\"odinger operators, see e.g.~the accounts in
\cite{CyconFKS-87,CarmonaL-90,PasturF-92,Stollmann-01,Veselic-06b}. In
a subsequent paper we will study the H\"older continuity of the IDS
for quantum graphs with randomly perturbed lengths of edges.
See~\cite{HelmV} for a Wegner estimate for alloy type potentials on metric graphs
and \cite{hislop-post:pre06,ExnerHS} for results on localization for
certain quantum graphs.

\emph{Localization} is the phenomenon, that certain quantum
Hamiltonians, describing disordered solid systems, exhibit pure point
spectrum almost surely.  Other ergodic operators exhibit purely
continuous spectrum, while it is conjectured, that for a large class
of operators pure point and continuous spectra should coexist, with a
(or several) sharp energy value separating them. This energy is called
\emph{mobility edge}.

Although the misconception that the IDS has a singularity of some kind
at the mobility edge was discarded by Wegner in \cite{Wegner-81},
there is still a strong relation between properties of the IDS and
localized states (corresponding to p.p.~spectrum).  Namely, the proof
of localization with the so far most widely applicable method, the
\emph{multi scale analysis} introduced by Fr\"ohlich and
Spencer~\cite{FroehlichS-83}, uses as a key ingredient an upper bound
on the \emph{density of states}. This function is the derivative of
the IDS and its existence (for certain models) may be proved by using
an estimate going back to Wegner \cite{Wegner-81}. 

\smallskip

For periodic operators in Euclidean geometry --- the most regular
form of ergodic Schr\"odinger operators --- the continuity of the IDS
is established under mild conditions on the potential, see
e.g.~\cite{Shen-02} and the references therein.

A substantial body of literature is devoted to randomly perturbed
periodic operators, where the perturbation is of alloy type. Under
certain conditions it is known that these random operators have also
an continuous IDS, i.e., that the random perturbation conserves the
continuity of the IDS. {}From the physical point of view it is
actually expected that the IDS of the random operators should be even
more regular than the one of periodic ones. However, only for certain
discrete models, better regularity of the IDS than continuity has been
proven, see for instance \cite{SimonT-85,ConstantinescuFS-84}.

For more general geometries than $\RR^d$ the situation is somewhat
different. In this situation even the periodic Laplace-Beltrami
operator (without any potential) on an abelian covering may have
$L^2$-eigenfunctions, as was already indicated in \cite{Sunada-88},
referring to an example in~\cite{KobayashiOS-89}. This is equivalent
to a discontinuity of the IDS (cf.~Proposition~\ref{p-IDS}).  Other
cases with jumps in the IDS are given by quasi-crystals
\cite{KlassertLS-03,LenzS}, periodic operators on covering graphs and
percolation Hamiltonians \cite{Veselic-05a,Veselic-05b}, random
necklace models \cite{KostrykinS-04} or fourth order differential
operators, see e.g.~\cite{Kuchment-93}. However, in particular cases,
the continuity of the IDS of periodic Schr\"odinger operators on an
abelian covering manifold can be established using a criterion of
Sunada (cf.~\cite{Sunada-90}).

Our main results are Wegner estimates for particular random
perturbations of periodic operators on manifolds
(cf.~Theorems~\ref{t-RAPWE} and~\ref{t-RAMWE}). The perturbation is
assumed to be of alloy-type and may enter the operators via the
\emph{potential} or the \emph{metric}, defined in the models~RAP
and~RAM (see Definitions~\ref{d-RAP} and~\ref{d-RAM}).  These
estimates imply the continuity of the IDS, even if the unperturbed,
periodic operator had a discontinuous IDS. Thus, while alloy type
perturbations \emph{preserve} the continuity of the IDS in the
Euclidean case, they are even IDS-continuity \emph{improving} for
certain operators on manifolds.

The paper is organized as follows: In the following section we
introduce our models RAP and RAM and state the main results.
Section~\ref{perOp} is devoted to periodic operators with abelian
covering group.  In Sections~\ref{s-RAPWE} and~\ref{s-RAMWE} we prove
Wegner estimates for both models RAP and RAM, respectively. For this
aim, we need a (super) trace class estimate of an effective
perturbation in each model (see Propositions~\ref{p-Veff-RAP}
and~\ref{p-Veff-RAM}).  The proof for this trace class estimate is
given in Sections~\ref{s-Trace}--\ref{s-TraceRes}. In the appendix we
provide necessary uniform results on Sobolev spaces on families of
manifolds which are used throughout this article.


\section{Model and results}
\label{ModRes}

Throughout the paper we will consider the following \emph{geometric
  situation}:

Let $(X,g_0)$ be a Riemannian manifold with a smooth metric $g_0$ and
$\Gamma$ an group acting freely, cocompactly and properly
discontinuously by isometries on $(X,g_0)$ such that the quotient
$M=X/\Gamma$ is a compact Riemannian manifold of the same dimension as
$X$. The stated assumptions imply that $\Gamma$ is a finitely generated group.
Typically, $X$ will be non-compact and thus $\Gamma$ infinite.

Let $(\Omega, \cB_\Omega, \PP)$ be a probability space on which
$\Gamma$ acts ergodically by measure preserving transformations
$\gamma \colon \Omega \to \Omega, \gamma \in \Gamma$, i.e., any
$\Gamma$-invariant set $B \in \cB_{\Omega}$ ($\gamma B = B$ for all
$\gamma \in \Gamma$) has probability $0$ or $1$. The expectation with
respect to $\PP$ is denoted by $\EE$.

We will be given two types of random objects over $(\Omega,
\cB_\Omega, \PP)$. The first is a family of random potentials on $X$,
the second is a family of random metrics. Put together, they will give
rise to a family of random operators whose study is our primary
concern here. Note that this includes the case that $\Omega$ contains
only one element and thus the operator family consists of a single
periodic operator.

As for the random geometry, the manifold $X$ is equipped with a family
of metrics $\{g_\omega\}_{\omega\in\Omega}$ with corresponding volume
forms $\vol_\omega$. With respect to a fixed periodic metric $g_0$, we
define a smooth section $A_\omega$ in the bundle $L(TX) \cong T^*\!X
\otimes TX$ via
\begin{equation*}
  g_\omega(x)(v,v)=g_0(x)(A_\omega(x)v,v)
\end{equation*}
for all $x \in X$, $v \in T_xX$ and $\omega \in \Omega$. In the
sequel, we will often suppress the dependence on $x \in X$.
We denote by $\Delta_\omega$ the 
non-positive Laplace operator on the mannifold $(X,g_\omega)$.

We need the following definition:

\begin{dfn} 
  \label{def:rel.bdd} We say that a family $\{g_\omega\}_\omega$ of
  metrics on $X$ is \emph{relatively bounded w.r.t.} the metric $g_0$
  on $X$ if for each $k \in \NN$ there are constants $C_{\rel,k} > 0$
  such that
  \begin{equation}
  \label{quasiisom}
    C_{\rel,0}^{-1}\, g_0(v,v) \le
    g_\omega(v,v)=g_0(A_\omega v,v) \le
    C_{\rel,0} \, g_0(v,v)
  \end{equation}
  for all $v \in TX$ and
  \begin{equation} \label{gradbound}
    \vert \nabla_0^k A_\omega(x) \vert_0 \le C_{\rel,k}
  \end{equation}
  for all $x \in X$ and all $\omega \in \Omega$.  Here
  $\nabla_0^k$ denotes the iterated covariant derivative w.r.t $g_0$
  and $| \cdot |_0$ is the (pointwise) norm w.r.t $g_0$ in the
  appropriate tensor bundle of $T^*\!X$ and $TX$. 
\end{dfn}

Since the periodic manifold $(X,g_0)$ is of bounded geometry, the
relative boundedness of the family $\{g_\omega\}_\omega$ implies that
$(X,g_\omega)$ is also of bounded geometry with constants $(r_0,C_k)$
\emph{independent} of $\omega$, as shown in Lemma~\ref{lem:rel.bd}.

Note that the lower bound in \eqref{quasiisom} implies that we have in
analogy to \eqref{gradbound} also a uniform bound on the derivatives
of $A^{-1}_\omega$, more precisely $\vert \nabla_0^k A_\omega^{-1}(x)
\vert_0 \le \wt C_{\rel,k} $.

The functions $x \mapsto (\det (A_\omega(x)))^{1/2}$ are positive,
smooth functions and satisfy
\begin{equation*}
  \int_X f(x) \bigl(\det (A_\omega(x))\bigr)^{1/2}\dd\vol_0(x) =
  \int_X f(x)  \dd\vol_\omega(x),
\end{equation*}
i.e., they are densities of the measures $\dd\vol_\omega$ with respect
to the unperturbed measure $\dd\vol_0$.  Consequently, for any
measurable subset $\Lambda\subset X$ and any pair $\omega^1,
\omega^2\in\Omega$ the operators
\begin{equation}
  \label{e-Somega}
  S_{\omega^1,\omega^2} \colon
        L^2(\Lambda,g_{\omega^1}) \to
        L^2(\Lambda,g_{\omega^2}), \quad
  S_{\omega^1,\omega^2}(f) =
  \bigl(\det (A_{\omega^1}A_{\omega^2}^{-1})\bigr)^{1/2}f
\end{equation}
are unitary, see also~\cite{LenzPV-04}. These operators will be used
in Section~\ref{s-RAMWE} to transform different Laplace-Beltrami
operators into the same Hilbert space.

The following conditions will be assumed throughout this paper:
\begin{asp}
  \label{main.ass}
  We assume that the family $\{g_\omega\}_\omega$ is jointly
  measurable, i.e., that $(\omega,v) \to g_\omega(x)(v,v)$ is
  measurable on $\Omega\times TX$.  In addition, we suppose that
  $\{g_\omega\}_\omega$ is relatively bounded in the sense of
  Definition \ref{def:rel.bdd} with respect to a fixed periodic metric
  $g_0$. Furthermore, we assume that the metrics are compatible in the
  sense that the covering transformations
  \begin{equation}
    \label{eq:covering.trafo}
    \gamma\colon (X,g_\omega) \to (X,g_{\gamma \omega}),
      \quad \gamma \colon x \mapsto \gamma x
  \end{equation}
  are isometries. Hence, the induced maps
  \begin{equation*}
    U_{(\omega,\gamma)}\colon
       L^2(X,g_{\gamma^{-1}\omega}) \to L^2(X,g_\omega), \qquad
    (U_{(\omega,\gamma)} f)(x) = f(\gamma^{-1}x)
  \end{equation*}
  are unitary operators
  between $L^2$-spaces over the manifolds
  $\{(X,g_\omega)\}_{\omega\in\Omega}$.

  As for the random potentials, let $V \colon \Omega \times X
  \longrightarrow [0,\infty[$ be jointly measurable and such that
  $V_\omega :=V(\omega,\cdot)$ is for all $\omega \in \Omega$
  relatively $\Delta_\omega$-bounded with relative bound strictly less
  than one.  Assume furthermore that $V(\gamma \omega, x)= V(\omega,
  \gamma^{-1} x)$ for arbitrary $x\in X$ and $\omega \in\Omega$.
\end{asp}

Given Assumption~\ref{main.ass}, we can now introduce the
corresponding random Schr\"odinger operator as
\begin{equation}
  \label{e-Homega}
   H_\omega := -\Delta_\omega + V_\omega \ge 0.
\end{equation}
In fact, these operators are defined by means of quadratic forms. For
more details we refer the reader to \cite{LenzPV-04}.  The operators
\eqref{e-Homega} satisfy the \textit{equivariance condition}
\begin{equation} \label{compcomp}
  H_\omega = U_{(\omega,\gamma)} H_{\gamma^{-1} \omega} U_{(\omega,\gamma)}^*,
\end{equation}
for all $\gamma \in \Gamma$ and $\omega \in \Omega$. Moreover, they
form a measurable family of operators in the sense of the next
definition as has been shown in Theorem~1 of~\cite{LenzPV-04}.

\begin{dfn}
  A family of selfadjoint operators $\{H_\omega\}_\omega$, where the
  domain of $H_\omega$ is a dense subspace of $L^2(D, g_\omega)$,
  is called {\em measurable family of operators} if

  \begin{equation}\label{weakmeas}
  \omega\mapsto \la f_\omega(\cdot), F(H_\omega) f_\omega(\cdot)\ra_\omega
  \end{equation}
  is measurable for all $F\colon \RR\rightarrow \CC$ bounded and
  measurable and all $f \colon \Omega\times X\rightarrow \RR$
  measurable with $f_\omega(\cdot) \in L^2(X,g_\omega)$ for every
  $\omega\in \Omega$.
\end{dfn}

This notion of measurability is consistent with the works of Kirsch
and Martinelli \cite{KirschM-82a,KirschM-82b} as discussed in
\cite{LenzPV-04}.  \medskip

A key object in our study is the \emph{integrated density of states}
(IDS). It will be defined next. Let $\cF \subset X$ be a
fundamental domain of $\Gamma$. We will need restrictions of operators
to agglomerates of translates of $\cF$.  For a finite set $I \subset
\Gamma$ define the agglomerate $\Lambda_0(I)$ of fundamental domains
associated with $I$ by
\begin{equation}
  \label{def:agglo}
  \Lambda_0(I) :=
  \bigcup_{\gamma\in I} \gamma \bar \cF \subset X.
\end{equation}

For technical reasons (e.g., the Sobolev extension
Theorem~\ref{thm:ext.op}), it is easier to work with a ``smoothed''
version $\Lambda(I)$ of the agglomerate $\Lambda_0(I)$, satisfying
$\Lambda(\gamma I)=\gamma \Lambda(I)$, and for some fixed radius $r>0$
the relation
\begin{equation}
  \label{e-agglo-incl}
  \Lambda_0(I) \subset \Lambda(I) \subset B_r(\Lambda_0(I)),
\end{equation}
where $B_r(A)$ denotes the open $r$-neighborhood of the set $A \subset
X$ with respect to the metric $g_0$. The construction of $\Lambda(I)$
is given in~\cite[pp.~593]{Brooks-81}. We show in
Lemma~\ref{lem:rel.bd.aggl} that $(\Lambda(I),g_0)$, and also
$(\Lambda(I),g_\omega)$, are of bounded geometry in the sense of
Definition~\ref{def:bdd.geo}, \emph{uniformly in $I$ and $\omega$}.

The restriction of $H_\omega$ to $\Lambda(I)$ with Dirichlet boundary
conditions is denoted by $H_\omega^I=H^{\Lambda(I)}_\omega$. The
corresponding spectral projections are denoted by $P_\omega^I$, i.e.
\begin{equation*}
   P_\omega^I \bigl({]}{-}\infty, E{[} \bigr):=
   \chi_{{]}{-}\infty, E{[}} (H_\omega^I)
\end{equation*}
and similarly $P_\omega\bigl({]}{-}\infty, E{[} \bigr):= \chi_{{]}{-}\infty, E{[}}(H_\omega)$.
We define the distribution function $ N_\omega^I$ on $\RR$ for
$H^I_\omega$ by
\begin{equation*}
  N_\omega^I (E) :=
      \frac 1{\vol_\omega \Lambda(I)} \Tr P_\omega^I \bigl({]}-\infty, E{[}\bigr).
\end{equation*}

The integrated density of states of the random operator
$\{H_\omega\}_\omega$ is defined as the distribution function
\begin{equation}
  \label{e-absIDS}
  N\colon \RR \to {[}0, \infty{[}, \quad N(E) :=
  \frac 1 {\EE[\vol_\bullet \cF] }
         \EE \Bigl[ \Tr\bigl( \chi_{\cF} \cdot P_\omega\bigl({]}{-}\infty, E{[} \bigr) \cdot \chi_\cF \bigr) \Bigr],
\end{equation}
where $\Tr$ is the trace in $L^2(\cF)$.

If the group $\Gamma$ is \emph{amenable} there exists a \emph{tempered
  F{\o}lner sequence}, i.e., an increasing sequence of finite,
non-empty subsets $I_l \subset\Gamma$, $l \in \NN$, with ``small
boundary'' cf.~\cite{Adachi-93,Lindenstrauss-01,LenzPV-04} for
details.

Theorem 4 in \cite{LenzPV-04} can be phrased as follows:

\begin{thm}
  \label{t-exhaustionIDS}
  Suppose that the transformation group $\Gamma$ of the covering $X
  \to M$ is amenable. Then at all continuity points $E$ of $N$ and for
  almost every $\omega$ the following convergence holds
  \begin{equation*}
    \lim_{l\to\infty} N_\omega^{I_l} (E) =
     \lim_{l\to\infty}\EE \bigl[ N_\bullet^{I_l} (E)\bigr] = N(E).
  \end{equation*}
\end{thm}

\begin{rem}
  \label{r-absIDS}
  Note that in~\cite{LenzPV-04} we proved the above theorem for the
  \emph{non}-smoothed domains $\Lambda_0(I)$, but the statement is
  still true for the smoothed domains $\Lambda(I)$. More precisely, if
  \begin{equation*}
    N_{0,\omega}^I(E) :=
      \frac 1{\vol_\omega \Lambda_0(I)}
      \Tr \chi_{{]}{-}\infty, E{[}} (H_\omega^{\Lambda_0(I)})
  \end{equation*}
  denotes the distribution function with respect to the domain
  $\Lambda_0(I)$, then domain monotonicity implies that
  \begin{equation*}
    N_{0,\omega}^{I_l}(E) \le
    \bigg(1 +
        \frac{\vol_\omega \bigl(\Lambda(I_l)\setminus \Lambda_0(I_l)\bigr)}
             {\vol_\omega\Lambda_0(I_l)}
    \bigg)N_\omega^{I_l}(E),
  \end{equation*}
  \sloppy where $H_\omega^{\Lambda_0(I_l)}$ denotes the Dirichlet
  operator on the (unsmoothed) agglomerate $\Lambda_0(I)$. The F\o
  lner property and~\eqref{quasiisom} now immediately imply that
  $\lim_{l \to \infty} N_{0,\omega}^{I_l}(E) \le \liminf_{l \to \infty}
  N_\omega^{I_l}(E)$.

  As for the converse inequality, we define a sequence $J_l \subset
  \Gamma$ as $J_l := I_l A$ with
  \begin{equation*}
    A := \{ \gamma \in \Gamma \, \mid \,
          d_0( \gamma \overline \cF, \overline \cF) \le r \}.
  \end{equation*}
  Note that $A$ is a finite set and that $J_l$ form also a tempered
  F\o lner sequence. Furthermore we have by construction that
  $\Lambda(I_l) \subset B_r(\Lambda_0(I_l)) \subset \Lambda_0(J_l)$,
  see~\cite[Proof of Lem.~2.4]{PeyerimhoffV-02}. As before, we derive
  $\limsup_{l \to \infty} N_\omega^{I_l}(E) \le \lim_{l \to \infty}
  N_{0,\omega}^{J_l}(E)$. The latter limit equals $N(E)$ since
  Theorem 4 in \cite{LenzPV-02} states that for \emph{any} tempered F\o lner
  sequence $(J_l)_l$ the finite volume approximations   $N_{0,\omega}^{J_l}$
  converge to $N$.
\end{rem}
Now we impose \emph{specific assumptions} to describe various
situations where we can say something about the continuity of the IDS.
We first study the case that
\begin{equation}
  \label{e-periodic}
  H \equiv H_\omega \text{ is a single $\Gamma$-periodic operator} .
\end{equation}
This fits in the general framework of ergodic operators if $\Omega$
contains a single element $\omega$.

The following is a basic result in the spectral analysis of
$\Gamma$-periodic operators.  Recall that a measure $\mu$ is called a
\emph{spectral measure} for the selfadjoint operator $H$ if, for a
Borel-measurable subset $B$ of $\RR$, the spectral projection
$\chi_B(H)=0$ if and only if $\mu(B)=0$.

{}From~\cite{LenzPV-02} or~\cite{LenzPV-04} we infer

\begin{prp}
  \label{p-spectralmeasure}
  Assume that $H$ is $\Gamma$-periodic.  Then the IDS of $H$ defined
  in~\eqref{e-absIDS} is the distribution function of a spectral
  measure for $H$. In particular, the IDS is continuous at $E$ if and
  only if $E\in\RR$ is not an eigenvalue of $H$.
\end{prp}

If one additionally assumes that the group $\Gamma$ is \emph{abelian},
much more is known.  
In fact, for abelian groups $\Gamma$, strong regularity properties of the
IDS are established in results of Sunada~\cite{Sunada-90} and
Gruber~\cite{Gruber-02}. Sunada proves that under a certain additional
assumption the spectrum has no point component.  Gruber shows that the
spectrum has no singularly continuous component. Putting this together
one obtains the following result.
\begin{thm} 
  \label{t-SunadaGruber} 
  Assume that $H$ is $\Gamma$-periodic and let $X$ be the maximal
  abelian covering of a closed Riemannian manifold $M$. If the
  potential $V$ is smooth and~$M$ admits a nontrivial $S^1$-action
  whose generating vector field is parallel, then $H$ has purely
  absolutely continuous spectrum and, consequently, the IDS is
  absolutely continuous.
\end{thm}

\begin{exm}
  In~\cite{Sunada-90} Sunada considers, as a particular example of a
  manifold $M$ which satisfies the assumptions of Theorem
  \ref{t-SunadaGruber}, a Riemannian product of a flat torus and a
  closed manifold.
\end{exm}

A more detailed discussion of $\Gamma$-periodic operators for abelian
$\Gamma$ can be found in Section~\ref{perOp}. After this discussion of
the periodic case, we will now deal with instances of random
operators.  The key result in our analysis of continuity properties of
the IDS will be the Wegner estimates discussed below.  The first type
of random operators is given in the following

\begin{dfn}
  \label{d-RAP} A family of operators $\{H_\omega\}_\omega$ as in
  \eqref{e-Homega} is called \emph{\textbf{r}andom Schr\"o\-din\-ger
    operator with \textbf{a}lloy type \textbf{p}otential } and
  abbreviated by RAP if it satisfies the following conditions:
  \begin{enumerate}
  \item[(P1)] Let $q_\gamma\colon \Omega\to {[}0, \infty{[}$, $\gamma
    \in \Gamma$ be a collection of i.i.d.~random variables, whose
    distribution measure $\mu$ has a compactly supported bounded
    density $f$ with $\supp f \subset [q_-,q_+] \subset [0,\infty{[}$.
  \item[(P2)] Let the function, $v\colon X \to \RR$, called \emph{single site potential}, satisfy
    \begin{equation}
    \label{e-lambda}
       v \ge \lambda \chi_\cF, \quad
       v \in L_{\mathrm c}^{p(d)}(X,g_0),
    \end{equation}
    where $\lambda$ is some positive real, 
    \begin{equation*}
      p(d) \ge  2 \quad  \text{if} \quad  d \le 3
            \qquad \text{and} \qquad
      p(d)> d/2   \quad  \text{if} \quad  d \ge  4.
    \end{equation*}
  \item[(P3)] Define the family of potentials by
    \begin{equation*}
      V_\omega(x) =
      V_{\per}(x)+\sum_{\gamma\in\Gamma} q_\gamma(\omega)
      v(\gamma^{-1}x)
   \end{equation*}
    with $V_{\per} \ge 0$ a bounded periodic potential.
  \item[(P4)] Let $\{g_\omega\}_\omega$ be a random metric, relatively
    bounded with respect to $g_0$, which is independent of the random
    variables $\{q_\gamma\}_\gamma$.
  \end{enumerate}
  Here, the random variable $q_\gamma$ is called \emph{coupling constant} and
  $\omega\in \Omega$ a \emph{random configuration}.
\end{dfn}
Note that random Schr\"odinger operators with alloy type potentials
satisfy the conditions~\eqref{compcomp} and~\eqref{weakmeas}.

\begin{rem}
  \label{r-pot}
  Due to the assumptions in Definition~\ref{d-RAP}, the potential
  $V_\omega$ is \emph{uniformly} infinitesimally
  $\Delta_\omega$-bounded in $\omega \in \Omega$, i.e., for every
  $\eps>0$ there exists a constant $C_\eps>0$ \emph{independent} of
  $\omega$ such that
  \begin{equation*}
    \| V_\omega u \| \le
    \eps \| \Delta_\omega u\| + C_\eps \| u\|
  \end{equation*}
  where $\|\cdot\|$ denotes the norm on $L^2(X,g_0)$ with respect to
  the periodic metric~$g_0$.

  All our results hold if we replace the condition $V_\per \ge 0$ by
  $V_\per\ge c$ for some negative number $c$. However, it is crucial
  that the single site potential $v$ does not change sign.
\end{rem}

The corresponding Wegner estimate proved in Section~\ref{s-RAPWE}
reads as follows.

\begin{thm}[Wegner estimate for RAP]
  \label{t-RAPWE}
  Let $\{H_\omega\}_\omega$ be as in Definition~\ref{d-RAP}.  Then,
  for all $p>1$, and $E \in \RR$, there exists $C_{E,p}>0$ such that
  for all $\epsilon \in {[}0, 1/2 {[}$, and $I \subset \Gamma$ finite,
  \begin{equation}
    \EE \bigl[ \Tr \bigl(
            P_\bullet^I([E-\epsilon, E+\epsilon])\bigr)\bigr]
       \le C_{E,p} \, \epsilon^{1/p} \, (\#I^+),
  \end{equation}
  where $I^+=I^+(v)$ abbreviates
  \begin{equation}
    \label{eq:def.i+}
    I^+(v) := \{ \gamma \in \Gamma \mid (\supp v\circ\gamma^{-1}) \cap
    \Lambda(I) \neq \emptyset \}.
  \end{equation} 
  The constant $C_{E,p}$ depends only on $E$, $p$, the manifold $(X,g_0)$, 
  the constants $C_{\rel,k}, k \in \NN$ of Definition~\ref{def:rel.bdd}, 
  the group $\Gamma$, the fundamental domain $\cF$, the single site potential
  $v$, the supremum of the density $\|f\|_\infty$ and its support $\supp f$ and
  $\|V_\per\|_\infty$.
\end{thm}

The proof of Theorem \ref{t-RAPWE} shows that the same result holds
true if we replace the condition~\eqref{e-lambda} by the following
weaker assumption
\begin{equation}
  \label{e-lambda-weak}
  v \ge 0, \quad v \in L_{\mathrm c}^{p(d)}(X,g_0)
       \qquad \text{and}\qquad
  \sum_{\gamma\in \Gamma} v \circ \gamma^{-1} \ge \lambda
   \quad \text{on $X$.}
\end{equation}

Theorems~\ref{t-exhaustionIDS} and~\ref{t-RAPWE} immediately imply:
\begin{cor}[H\"older continuity of the IDS for RAP]
  Under the assumptions of Theorem~\ref{t-RAPWE} and the amenability
  of $\Gamma$ the integrated density of states is locally H\"older
  continuous on $\RR$, for any H\"older exponent $1/p$ strictly
  smaller than $1$.
\end{cor}

Now, we consider Laplacians with random metrics as defined in
\begin{dfn}
  \label{d-RAM}
  A family of operators $\{H_\omega\}_\omega$ as in \eqref{e-Homega}
  is called \emph{\textbf{r}andom Laplace operator with \textbf{a}lloy
    type \textbf{m}etric } and abbreviated by (RAM) if it satisfies
  the following conditions:
  \begin{enumerate}
  \item[(M1)] Let $r_\gamma\colon \Omega\to \RR$, $\gamma \in \Gamma$ be
    a collection of i.i.d.~random variables, whose distribution
    measure $\nu$ has a compactly supported density $h$ of bounded
    variation.
  \item[(M2)] Let $u\in C_{\mathrm c}^\infty(X)$ with $u \geq \kappa
    \chi_\cF$ and $\kappa >0$ be given.
  \item[(M3)] Define the family of conformally perturbed Riemannian
    metrics on $X$ for $\omega\in \Omega$:
    \begin{equation}
    \label{atm}
      g_\omega(x):=a_\omega(x) g_0(x):=
        \Big(\sum_{\gamma\in \Gamma}\, \mathrm e^{r_\gamma(\omega)} \,
           u(\gamma^{-1} x) \Big) g_0(x).
    \end{equation}
  \item[(M4)] Let $V_\omega$ be identically zero, i.e.~$H_\omega=
    -\Delta_\omega$ for all $\omega \in \Omega$.
  \end{enumerate}
  In this situation, the random variable $r_\gamma$ is called \emph{coupling
    constant}, $\omega\in \Omega$ a \emph{random configuration} and
  $u$ the \emph{single site deformation}.
\end{dfn}

Note that the family $\{g_\omega\}_\omega$ is relatively bounded with
respect to $g_0$ and that the constants $C_{\rel,k}, k \in \NN$ depend only on
$u$, its derivatives and $\supp h$.

To state our Wegner estimate for alloy type metrics, whose proof is
given in Section~\ref{s-RAMWE}, we need one more piece of notation.
Namely, for $a\ge 1$, set $J_a =[1/a,a]$.

\begin{thm}[Wegner estimate for RAM]
  \label{t-RAMWE}
  Let $\{H_\omega\}_\omega$ with alloy-type metric be given.  Then,
  for every $ p>1$, $a \ge 1$ there exists $C_{a,p} >0$ such that for
  every finite $I\subset \Gamma$
  \begin{equation}
    \EE \bigl[ \Tr \bigl (P_\bullet^I ([E-\epsilon, E+\epsilon])
             \bigr)\bigr]
    \le C_{a,p} \, \epsilon^{1/p} \, (\# I^+)
  \end{equation}
  whenever $\epsilon < 1/2$ and $[E-\epsilon, E+\epsilon]\subset J_a$.
  Here $I^+=I^+(u)$ as in~\eqref{eq:def.i+}. 
  The constant $C_{a,p}$
  depends on $a$, $p$, the manifold $(X,g_0)$, the group $\Gamma$, the
  fundamental domain $\cF$, the single site deformation $u$ and the
  density $h$.
\end{thm}
Note that this Wegner estimate in contrast to Theorem \ref{t-RAPWE} 
does not apply to a neighbourhood of the energy zero. An intuitive explanation
for this phenomenon is given in Example \ref{x-discontinous}.
Similarly as before, Theorems~\ref{t-exhaustionIDS} and~\ref{t-RAMWE}
imply:
\begin{cor}[H\"older continuity of the IDS for RAM]
  \label{DOSthm}
  Under the assumptions of Theorem~\ref{t-RAMWE} and the amenability
  of $\Gamma$ the integrated density of states is locally H\"older
  continuous on $\RR \setminus \{0\}$, for any H\"older exponent
  strictly smaller than $1$.
\end{cor}

In Example~\ref{x-discontinous} below, we mention a class of 
abelian, non-compact covering manifolds $(X,g_0)$, where the
corresponding Laplace operator has an $L^2$-eigenfunction with positive eigenvalue.
(Note that this phenomenon does not occur for periodic Schr\"odinger operators
on the Euclidean space.)  Consequently, the IDS of the periodic Laplace
operator has a discontinuity away from $0$, whereas the IDS of the
random family RAM is continuous away from $0$. In this example, the
introducion of randomness improves the regularity of the IDS.


\section{Periodic operators on manifolds}
\label{perOp}

In this section, we consider a covering manifold $X$ with
\emph{abelian} covering group $\Gamma$ and $\Gamma$-periodic metric
$g$. In this case, all irreducible, unitary representations are
one-dimensional. Therefore, the set of their equivalence classes forms
a group, called the \emph{dual group} of $\Gamma$ and denoted by
$\hat\Gamma$.  In particular, if $\Gamma=\ZZ^r$, we have $\hat
\Gamma=\TT^r$, the $r$-dimensional torus together with its Haar
measure denoted by $\dd \theta$.

A periodic Schr\"odinger operator $H$ admits a direct integral
decomposition
\begin{equation*}
  UHU^* = \int\nolimits_{\hat \Gamma}^\oplus H^\theta \dd \theta
\end{equation*}
where $U$ is a unitary map from $L^2(X) \cong \ell^2(\Gamma) \otimes
L^2(\cF)$ onto $\int_{\hat \Gamma}^\oplus L^2(\cF) \dd \theta \cong
L^2(\hat \Gamma) \otimes L^2(\cF)$ acting as a \emph{partial Fourier
  transformation} on the group part.  The operators $H^\theta$,
$\theta \in \hat\Gamma$ are defined on the set of $\theta$-periodic
functions, i.e., functions $\psi$ such that $\psi(\gamma
x)=\theta(\gamma)\psi(x)$ for all $x \in X$ and $\gamma \in \Gamma$.
It suffices to consider such functions on a fundamental domain $\cF$.
Since $H^\theta$ can be considered as an elliptic operator on a
complex line bundle of a compact manifold (cf.~the notion of a
``twisted'' Laplacian in \cite{Sunada-88}), it has purely discrete
spectrum for all $\theta \in \hat \Gamma$.  We denote by $E_1(\theta)
\le E_2(\theta) \le \dots$ the eigenvalues of $H^\theta$ in
non-decreasing order and including multiplicities.

Let us start with the following proposition giving a formula to
calculate the IDS:
\begin{prp} 
  \label{p-ids-nc}
  Let $\Gamma$ be an abelian group and $H$ be
  $\Gamma$-periodic and
  \begin{equation*}
    N(E) := \frac 1 {\vol \cF}
       \Tr \bigl( \chi_{\cF} \cdot
            \chi_{{]}-\infty,E{[}} (H) \cdot \chi_{\cF} \bigr),
  \end{equation*}
  where $\Tr$ is the trace in $L^2(\cF)$.  Then
  \begin{equation*}
    N(E) = \frac 1 {\vol \cF}
      \int\limits_{\hat \Gamma}
         \Tr \chi_{{]}{-}\infty,E{[}}(H^\theta) \dd \theta
        = \frac 1 {\vol \cF }
             \sum_n \meas \{ \theta \in \hat \Gamma \mid
                              E_n(\theta) < E \}.
  \end{equation*}
\end{prp}

\begin{proof}
  Let $\{\phi_n\}$ be an orthonormal basis of $L^2(\cF)$ and $\wt
  \phi_n$ the trivial extension of $\phi_n$ in $L^2(X)$. Then $\wt
  \phi_n = \delta_e \otimes \phi_n$ via the identification $L^2(X)
  \cong \ell^2(\Gamma) \otimes L^2(\cF)$.  We have
  \begin{multline*}
    (\vol \cF) \, N(E) =
    \Tr \bigl(\chi_\cF \cdot \chi_{{]}-\infty,E{[}} (H)
                        \cdot \chi_{\cF} \bigr)\\
    = \sum_n \langle \wt \phi_n,
          \chi_{{]}-\infty,E{[}}(H) \wt \phi_n
                         \rangle_{L^2(X)}\\
    =  \sum_n \langle U \wt \phi_n,
                      U \chi_{ {]}-\infty,E{[}}(H) U^* U \wt \phi_n
           \rangle_{L^2(\hat \Gamma) \otimes L^2(\cF)}.
   \end{multline*}
   Here, $U \wt \phi_n = U(\delta_e \otimes \phi_n) = \1 \otimes
   \phi_n$ where $\1$ is the constant function on $\hat \Gamma$ and
   \begin{equation*}
     U \chi_{{]}-\infty,E{[}}(H) U^* =
     \int_{\hat \Gamma}^\oplus \chi_{]-\infty,E{[}}(H^\theta) \dd \theta.
   \end{equation*}
   Therefore, $(\vol \cF) N(E)$ equals
   \begin{multline*}
     \sum_n \int_{\hat \Gamma} \langle \phi_n,
     \chi_{]-\infty,E{[}}(H^\theta) \phi_n \rangle_{L_2(\cF)} \dd
     \theta =
     \int_{\hat \Gamma} \Tr \chi_{]-\infty,E{[}}(H^\theta)
     \dd \theta \\=
     \sum_n \int_{\hat \Gamma} \langle \phi_n^\theta,
       \chi_{]-\infty,E{[}}(H^\theta) \phi_n^\theta \rangle_{L_2(\cF)}
             \dd \theta,
   \end{multline*}
   by Fubini, where $\{\phi_n^\theta\}_n$ is an orthogonal basis of
   eigenfunctions of $H^\theta$ associated to $E_n(\theta)$ in each
   fiber. But the latter integral equals the measure of $\{ \theta \in
   \hat \Gamma \mid E_n(\theta)<E\}$ and the result follows.
\end{proof}

\begin{prp}
  \label{p-IDS}
  Let $\Gamma$ be an abelian group and $H$ be
  $\Gamma$-periodic.  Then the following assertions are equivalent:
  \begin{enumerate}[\rm(a)]
  \item $E\in\RR$ is an eigenvalue of $H$.
  \item $N(\cdot)$ is not continuous at $E$.
  \item There is an $n \in \NN$ such that $\meas\{ \theta \in \hat
    \Gamma \mid E_n (\theta) =E \} >0$.
  \end{enumerate}
\end{prp}

\begin{proof}
  The equivalence of~(a) and~(b) follows from
  Proposition~\ref{p-spectralmeasure}. The equivalence of~(b) and~(c)
  is an immediate consequence of Proposition~\ref{p-ids-nc}.
\end{proof}

\begin{exm}
  \label{x-discontinous}
  In~\cite[Prop.~4]{KobayashiOS-89} a class of examples of periodic
  Laplacians on an infinite covering manifold with an
  $L^2$-eigenfunction is constructed using Atiyah's $L^2$-index
  theorem. 
This class includes in particular periodic Laplace operators on
  abelian covering manifolds, e.g., the principal spin-bundle of a
  connected sum of a $K3$-surface and a $4$-dimensional torus.

Note that the corresponding eigenvalue is strictly
  positive. This can be seen as follows: Brooks' Theorem
  \cite{Brooks-81} implies that the bottom of the spectrum is strictly
  positive for non-amenable groups. For the case of amenable groups
  the bottom of the spectrum equals $0$, but it cannot be an
  eigenvalue, which follows from~\cite{Sarnak-82b}
  and~\cite{Sullivan-87b}, cf.~\cite[Prop.~3]{Sunada-88}. 
\end{exm}


\section{Wegner estimate for alloy-type potentials}
\label{s-RAPWE}

Now we are in the position to prove Theorem~\ref{t-RAPWE} following
Wegner's original idea~\cite{Wegner-81} and using adaptations
from~\cite{Kirsch-96,Stollmann-00b,CombesHN-01}.

We will apply the Hellman-Feynman theorem, i.e., first order
perturbation theory, cf.~\cite{Kato-66} or~\cite{IsmailZ-88}, to the
purely discrete spectrum of $H_\omega^I$.  By
assumption~\eqref{e-lambda} in Definition~\ref{d-RAP} the derivatives
of the eigenvalues $E_n(\omega):= E_n^I (\omega) $ of $H_\omega^I$
obey
\begin{equation}
  \label{e-PosDer}
  \sum_{\gamma \in I^+}
          \frac \partial {\partial q_\gamma} E_n(\omega) =
  \sum_{\gamma \in I^+} \la \psi_n,(v\circ\gamma^{-1})\, \psi_n \ra \ge
  \lambda.
\end{equation}

Here $\psi_n$ denotes the eigenfunction corresponding to $E_n(\omega)$
and $I^+=I^+(v)$ as in~\eqref{eq:def.i+}. 

For $0< \varepsilon < 1/2$, let $\rho:=\rho_{E,\varepsilon}\colon \RR
\to [-1,0]$ be a smooth, monotone switch function, i.e., $\rho$
satisfies $ \rho\equiv -1$ on ${]}{-}\infty,E-\epsilon{]}$,
$\rho\equiv 0$ on ${[}E+\epsilon,\infty{[}$ and $\|\rho'\|_\infty \le
1/\epsilon$.  Then we have $\chi_{[E-\epsilon, E+\epsilon]} (x) \le
\int_{-2\epsilon}^{2\epsilon} \rho'(x+t) \dd t$ and thus by the
spectral theorem
\begin{equation*}
  P_\omega^I([E-\epsilon, E+\epsilon]) \le
  \int_{-2\epsilon}^{2\epsilon} \rho'(H_\omega^I+t) \dd t.
\end{equation*}
The chain rule implies
\begin{equation*}
  \sum_{\gamma \in I^+}
          \frac{\partial}{\partial q_\gamma}
          \rho(E_n(\omega)+t) =
   \rho'(E_n(\omega)+t)
         \sum_{\gamma \in I^+}
       \frac{\partial}{\partial q_\gamma} E_n(\omega)
\end{equation*}
which is by~\eqref{e-PosDer} bounded from below by $\lambda \,
\rho'(E_n(\omega)+t)$.  Thus we can divide by $\lambda >0$ and obtain
$\rho'(E_n(\omega)+t) \le \frac{1}{\lambda}\sum_{\gamma \in I^+}
\frac{\partial}{\partial q_\gamma}\rho(E_n(\omega)+t) $ and
consequently
\begin{equation*}
  \Tr \Big ( P_\omega^I([E-\epsilon, E+\epsilon]) \Big ) \le
     \frac 1 \lambda \int_{-2\epsilon}^{2\epsilon}
     \sum_{n \in \NN} \sum_{\gamma \in I^+}
           \frac{\partial}{\partial q_\gamma}\rho(E_n(\omega)+t)
        \dd t.
\end{equation*}
By our independence assumption on the various random ingredients of the model, 
the expectation value corresponds to 
an integration with respect to a product measure.
The averaging effect we need is produced by integration over a single coupling constant. 
Afterwards we take expectation over all the remaining randomness. 
\begin{multline*}
   \EE \bigl[ \Tr
             P_\bullet^I \bigl( [E-\epsilon, E+\epsilon] \bigr)
        \bigr] \le \\
   \frac 1 \lambda \int_{-2\epsilon}^{2\epsilon} \dd t \,
     \sum_{\gamma \in I^+}
        \EE \Bigl[ \int_{q_-}^{q_+}  f(q_\gamma)
              \sum_{n \in \NN} \frac{\partial}{\partial
                 q_\gamma}\rho(E_n(\bullet)+t)
                       \dd q_\gamma\Bigr].
\end{multline*}
Now, $\rho$ is increasing and $E_n$ is increasing in $q_\gamma$. Thus,
$\frac{\partial}{\partial q_\gamma}\rho(E_n(\omega)+t)\geq 0$ and the
modulus of the $\dd q_\gamma$-integral in the square brackets is
bounded by
\begin{equation*}
  \|f\|_\infty \int_{q_-}^{q_+} \sum_{n \in \NN}
             \frac{\partial\rho}{\partial q_\gamma}(E_n(\omega)+t)
               \dd q_\gamma=
     \|f\|_\infty \int_{q_-}^{q_+}
        \Tr \Bigl(
           \frac{\partial\rho}{\partial q_\gamma}(H_\omega^I+t)\Bigr)
         \dd q_\gamma.
\end{equation*}
Here, the integral on the right hand side is equal to
\begin{equation}
  \label{e-TrDiff}
  \Tr \bigl( \rho ( H_2+t ) - \rho ( H_1+t ) \bigr)
\end{equation}
where $H_1:= H_\omega^I +(q_--q_\gamma(\omega)) \cdot
(v\circ\gamma^{-1})$, $H_2:= H_1 +(q_+ - q_-) \cdot
(v\circ\gamma^{-1})$ and $q_-,q_+$ denote the two extremal values
which the random variable $q_\gamma$ may take.  Krein's trace
identity, see e.g.~\cite{BirmanY-93}, now tells us that
\eqref{e-TrDiff} equals
\begin{equation}
  \label{e-SSF}
  \int  \rho' (\wt E) \, \xi(\wt E, H_2+t, H_1+t) \dd \wt E
\end{equation}
where $\xi$ is the spectral shift function.

In the following definition we introduce a technical piece of notation 
which plays an important role in the sequel:
\begin{dfn}
  \label{d-constants}
  Let $p$ be the inverse H\"older exponent chosen in
  Theorem~\ref{t-RAPWE}. Let $0 < \alpha < 1$ be given by $1/p+\alpha
  =1$ and $q \in 2\NN$ be the smallest even integer satisfying $q \ge
  \max \{ 6, d/2+2\}$. Finally, $k$ denotes the smallest integer such
  that $k/q \ge 1/\alpha$ and $g(x):=(x+1)^{-k}$.
\end{dfn}

Since $H_\omega \ge 0$ for all $\omega$ the operator $g(H_\omega)$ is
well defined.  As discussed in Section~\ref{s-Trace}, $g(H_2)- g(H_1)$
is trace class and even belongs to $\cJ_{\alpha}$. Here,
$\cJ_{\alpha}$ denotes the (super) trace class ideal of compact
operators whose singular values are summable to the power $\alpha$.
This class of operators is discussed in more detail at the beginning of Section \ref{s-Trace}.
 Note that since $\alpha<1$ the ideal
$\cJ_\alpha$ is a subset of the trace class ideal.

The invariance principle, see e.g.~\cite{BirmanY-93}, tells us that
the modulus of the expression~\eqref{e-SSF} equals
\begin{equation*}
  \Bigl|\int\rho'(\wt E) \xi \bigl( g(\wt E-t), g(H_2), g(H_1) \bigr)
           \dd \wt E \Bigr|
\end{equation*}
The H\"older inequality for $ 1/p + \alpha =1$ gives an upper bound
\begin{equation*}
  \Bigl( \int (\rho'(\wt E))^p \dd \wt E \Bigr)^{1/p}
  \Bigl( \int_{\supp \rho'}
      \bigl|\xi(g(\wt E-t),g(H_2),g(H_1))\bigr|^{1/\alpha}
         \dd \wt E  \Bigr)^\alpha.
\end{equation*}
The first factor can be estimated by
\begin{equation*}
  \Bigl( \|\rho'\|_\infty^{p-1}
            \int \rho'(\wt E) \dd \wt E \Big)^{1/p}  \le
  \epsilon^{-1+1/p}
\end{equation*}
and the second obeys the upper bound,
\begin{equation*}
  \Bigl( \frac 1 k  (E+2)^{k+1}
       \int_\RR \bigl|\xi(E',g(H_2),g(H_1)) \bigr|^{1/\alpha}
          \dd E'   \Bigr)^\alpha.
\end{equation*}
By a result of~\cite{CombesHN-01},
\begin{equation}
\label{e-SuperTraceBd}
  \Bigl( \int_\RR
         \bigl|\xi(E',g(H_2),g(H_1))\bigr|^{1/\alpha}
          \dd E' \Bigr)^ \alpha \le
  \| g(H_2)-g(H_1) \|_{\cJ_\alpha}^\alpha.
\end{equation}
The operator $g(H_2)-g(H_1)$ appearing on the right side
of~\eqref{e-SuperTraceBd} is a kind of \emph{effective perturbation}.
To estimate its $\|\cdot\|_{\cJ_{\alpha}}$-norm we use the following
immediate consequence of Theorem~\ref{t-deffestP}. 

\begin{prp}
  \label{p-Veff-RAP}
  For given $p>1$ let $\alpha=1-1/p$, $k$ and $g$ be as in
  Definition~\ref{d-constants}.  Furthermore, let $H_1:= H_\omega^I
  +(q_--q_\gamma(\omega)) \, v\circ\gamma^{-1}$, $H_2:= H_1 +(q_+ -
  q_-) \cdot(v\circ\gamma^{-1})$ be as defined earlier in this
  section.  Then there exists a constant $C_\alpha$, which does not
  depend on $\omega$, $I$ and $\gamma$, such that
  \begin{equation}
    \|g(H_2)-g(H_1)\|_{\cJ_{\alpha}}\le C_\alpha.
  \end{equation}
\end{prp}

Collecting the estimate of this section we obtain the desired result:
\begin{equation}
  \label{e-ExplicitConst}
  \EE \bigl[ \Tr \bigl ( P_\bullet^I([E-\epsilon, E+\epsilon])
              \bigr)\bigr] \le
  \frac{4\, (C_\alpha)^\alpha  \|f\|_\infty} {k^\alpha \lambda }
     \, (E+2)^{\alpha(k+1)} \, (\# I^+) \, \epsilon^{1/p}.
\end{equation}



\section{Wegner estimate for alloy-type metrics}
\label{s-RAMWE}

This section is devoted to the proof and discussion of
Theorem~\ref{t-RAMWE}.  Without loss of generality we may assume
$\sum_\gamma u(\gamma^{-1} x)\equiv 1$ on $X$ by replacing simultaneously the single
site deformation $u(x)$ by $u(x)/\sum_{\gamma}u(\gamma^{-1}x)$ and
$g_0(x)$ by $g_0(x) \sum_\gamma u(\gamma^{-1} x)$.  In the sequel we
will tacitly identify $\Omega$ and $\bigtimes\nolimits_{\gamma\in
  \Gamma} \RR$ via $\omega = \{r_\gamma (\omega)\}_\gamma$.

The following lemma describes how eigenvalues are moved by a special
change of parameters in the random Hamiltonian.  It is an analogue of
estimate~\eqref{e-PosDer}.

\begin{lem}
  \label{l-Ableitung}
  Denote by $E_n(\omega)=E_n^I(\omega)$ the eigenvalues of
  $-\Delta_\omega^I$.  Then
  \begin{equation*}
     \sum_{\gamma\in I^+}
            \frac{\partial E_n(\omega)}{\partial r_\gamma}
       = - E_n(\omega)
  \end{equation*}
  for all $n \in \NN$, $\omega \in \Omega$, and $I \subset \Gamma$.
  Here $I^+=I^+(u)$ as in~\eqref{eq:def.i+}.
\end{lem}

\begin{proof}
  Since $g_{\omega + t(1,\dots,1)}{\restriction}_{\Lambda(I)} = \mathrm e^t
  g_\omega{\restriction}_{\Lambda(I)}$ for $(1, \dots, 1) \in \RR^{I^+}$,
the operator $\Delta_{\omega + t(1,\dots,1)}^I$ is a conformal perturbation 
of $\Delta_{\omega}^I$ with perfactor $\mathrm e^{-t}$. Hence
  \begin{equation*}
    E_n (\omega+t (1, \dots,1)) = \mathrm e^{-t} E_n(\omega).
  \end{equation*}
  This gives
  \begin{equation*}
    \sum_{\gamma\in I^+}
         \frac{\partial E_n(\omega)}{\partial r_\gamma} =
    \frac{\partial}{\partial t} \Bigl|_{t=0} \Bigr.
              E_n(\omega+t (1,\dots,1))
     = - E_n(\omega),
  \end{equation*}
  and the proof is finished.
\end{proof}

\begin{rem}
If we consider an eigenvalue which is bounded away from zero, the
lemma tells us that the absolute value of its derivative has a
positive lower bound. Thus it is ensured, that this eigenvalue is
moved by the chosen change in the coupling constants.  This approach
is analogous to the vector field method of \cite{Klopp-95a} and related to Wegner
estimates for multiplicative perturbations.
\end{rem}

We have to analyze the change of elements $\omega \in \Omega$ at
a single coordinate. To do so, we define $\theta_\gamma^s (\omega)$ for
$\gamma \in \Gamma$ and $s\in \RR$ by
\begin{equation*}
  \bigl(\theta_\gamma^s (\omega)\bigr)_\beta :=
  \begin{cases}
      \omega_\beta, & \text{for $\beta \neq \gamma$,} \\
      s,            & \text{for $\beta = \gamma$,}
  \end{cases}
\end{equation*}
i.e., the sequence $\theta_\gamma^s (\omega)$ coincides with $\omega$
up to position $\gamma$, where its value is $s$.

Let $\rho$ be as in Section~\ref{s-RAPWE}. Using
Lemma~\ref{l-Ableitung}, the chain rule, and the arguments from
Section~\ref{s-RAPWE}, we obtain for $E_n(\omega) \ge 1/a$
\begin{equation*}
0 \le \rho'(E_n(\omega)+t) 
\le 
a \Bigl ( -\sum_{\gamma\in I^+} \frac{\partial \rho}{\partial r_\gamma}(E_n(\omega)+t)\Bigr )
\end{equation*}
and thus the bound
\begin{equation*}
   \EE \bigl[ \Tr \bigl ( P_\bullet^I([E-\epsilon, E+\epsilon])
                \bigr)\bigr] \le a
     \int_{-2\epsilon}^{2\epsilon}
         \sum_{\gamma \in I^+} \EE ( T_\bullet (\gamma)) \dd t
\end{equation*}
with
\begin{equation*}
   T_\omega (\gamma) :=
    - \int  h(r_\gamma) \sum_{n \in \NN}
                \frac{\partial}{\partial r_\gamma}
         \rho(E_n(\omega)+t) \dd r_\gamma.
\end{equation*}
Let $\omega^1:= \theta_\gamma^0(\omega)$. Since $\omega^1$ does not
depend on $r_\gamma$, we can replace $\frac{\partial}{\partial
  r_\gamma} \rho(E_n(\omega)+t)$ by $\frac{\partial}{\partial
  r_\gamma}\rho(E_n(\omega)+t)-\frac{\partial}{\partial
  r_\gamma}\rho(E_n (\omega^1)+t) $. Such a normalisation 
is also used in \cite{HislopK-02}. Thus
\begin{align*}
  T_\omega(\gamma) & =
   - \int  h(r_\gamma) \frac{\partial}{\partial r_\gamma}
         \Bigl(\sum_{n \in \NN}\rho(E_n(\omega) +t)
         -\sum_{n \in \NN}\rho(E_n(\omega^1)+t) \Bigr)
            \dd r_\gamma \\&=
   - \int  h(r_\gamma) \frac{\partial}{\partial r_\gamma}
         \Bigl(\Tr_\omega \rho(H_\omega^I +t) - \Tr_1\rho(H_1+t)
             \Bigr) \dd r_\gamma.
\end{align*}
Here, $H_1:=H_{\omega^1}^I$, $\Tr_1$ denotes the trace in the space
$L^2(\Lambda(I), g_{\omega^1})$ and $\Tr_\omega$ denotes the trace in
the space $L^2(\Lambda(I), g_{\omega})$.  By partial integration for
functions of bounded variation, this can be bounded in modulus by
\begin{equation}
  \label{e-thetaMax}
  \|h\|_{\mathrm{BV}} \, \max_{ s \in \supp h}
     \Bigl| \Tr_\omega \rho(H_{ \theta_\gamma^s (\omega)
  }^I+t)-\Tr_1\rho(H_1+t) \Bigr|.
\end{equation}
That bounded variation regularity of the density function is sufficient 
in such a situation was laready noted in \cite{KostrykinV-06}.
Choosing $\widetilde{s}$ such that the maximum is attained and setting
$H_2 := H_{\omega^2}^I$, $\omega^2:= \theta_\gamma^{\widetilde{s}}
(\omega)$, we can finally bound $|T_\omega (\gamma)|$ by
\begin{equation*}
  \|h\|_{\mathrm{BV}} \Bigl| \Tr_2 \rho(H_2+t)-\Tr_1\rho(H_1+t) \Bigr|.
\end{equation*}
Here, $\Tr_2$ denotes the trace in the space $L^2(\Lambda(I),
g_{\omega^2})$.  To be able to apply the theory of the spectral shift
function, we want to transform the two operators $H_1$ and $H_2$ into
the same Hilbert space.  To do so, we use the operator
$S=S_{\omega^1,\omega^2}$ defined in~\eqref{e-Somega}.  It is a
multiplication operator given by
\begin{equation}
  \label{S}
   S \colon L^2(\Lambda(I),g_{\omega^1}) \to
         L^2(\Lambda(I),g_{\omega^2}), \quad
   S\phi(x) =
      \Bigl( \frac{a_{\omega^1}}{a_{\omega^2}}\Bigr)^{d/4}
                      \phi(x),
\end{equation}
with $a_\omega$ defined in~\eqref{atm}. Now both operators $\wt H_1 :=
SH_1S^*$ and $H_2$ act on the same Hibert space
$L^2(\Lambda(I),g_{\omega^2})$.  Since $S$ is unitary, we have
$\Tr_1\rho(H_1+t)=\Tr_2 \rho(\wt H_1+t)$.

Similarly as in Section~\ref{s-RAPWE} we can bound $|\Tr_2
[\rho(H_2+t) -\rho(\wt H_1+t)]|$ by
\begin{equation*}
   \epsilon^{-1+1/p} \,
       \Big( \frac{1}{k} (E+2)^{k+1}\Big)^\alpha
                \|g(H_2)-g(\wt H_1)\|_{\cJ_{\alpha}}^\alpha.
\end{equation*}

Again we are left to estimate the $\|\cdot\|_{\cJ_\alpha}$-norm of the
effective perturbation $g(H_2)-g(\wt H_1)$. This is provided by the
following direct consequence of Theorem~\ref{t-deffestM}. 

\begin{prp}
  \label{p-Veff-RAM}
  For given $p>1$ let $\alpha=1-1/p$, $k$ and $g$ be as in
  Definition~\ref{d-constants}.  Furthermore, let $\wt H_1:= S
  H_{\omega^1}^I S^*$ with $\omega^1:=\theta_\gamma^0(\omega)$ and
  $H_2 := H_{\omega^2}^I$ with $\omega^2:=\theta_\gamma^{\widetilde
    s}(\omega)$ be as defined above.  Then there exists a constant
  $\hat C_\alpha$, which does not depend on $\omega$, $I$ and
  $\gamma$, such that
  \begin{equation}
    \|g(H_2)-g(\wt H_1)\|_{\cJ_{\alpha}}\le \hat C_\alpha.
  \end{equation}
\end{prp}

Thus we obtain, for alloy type metrics, the Wegner estimate
\begin{equation}
  \EE \bigl[ \Tr \bigl (P_\bullet^I([E-\epsilon, E+\epsilon])
           \bigr)\bigr] \le
  \frac{4a \, (\hat C_\alpha)^\alpha\,
                  \|h\|_{\mathrm{BV}}}{k^\alpha}
          (E+2)^{\alpha(k+1)} \, (\#I^+)\, \epsilon^{1/p}.
\end{equation}

\begin{exm}
  While our Wegner estimate {\em for alloy type potentials} is valid
  for all bounded energy intervals, in the case of an {\em alloy type
    metric} we are only able to prove it for energy intervals away
  from zero.  Let us indicate the reason why our proof does not apply
  to low energies in the random metric case.

  For this we use a simplified example, where the probability space is
  the one-dimensional interval $[1,2]$, and the random operator the
  multiple of the Laplace operator on Euclidean space $H_s=-s\Delta$
  for $s \in[1,2]$.  The Fourier transformation of $H_s$ is $f(s,p):=s p^2$,
  and its derivative with respect to $s$ is $p^2$.  This
  derivative is positive, except for the value $p=0$.  This
  shows that moving the perturbation parameter $s$ smears out the
  spectrum of $H_s$ on any spectral subspace corresponding to energies
  away from zero. However, the effect of the perturbation parameter on
  a spectral subspace corresponding to energies around zero can be
  arbitrarily small.  A similar phenomenon occurs in
  Lemma~\ref{l-Ableitung} and thus in Theorem~\ref{t-RAMWE}.
\end{exm}


\section{Trace class bounds on the effective perturbations}
\label{s-Trace}

In this section we estimate certain effective perturbation operators,
which played a crucial role in Sections~\ref{s-RAPWE}
and~\ref{s-RAMWE}.  More precisely, we want to show that the effective
perturbations are in some (super) trace class spaces
$(\cJ_\alpha,\Vert \cdot \Vert_{\cJ_\alpha})$, and need to bound these
operators in the $\Vert\cdot \Vert_{\cJ_\alpha}$-topology. More
informations on (super) trace class spaces can be found, e.g.,
in~\cite{BirmanS-77,Simon-79,CombesHN-01}.

Let us start by shortly introducing the $\Vert\cdot
\Vert_{\cJ_\alpha}$-topology:
  For $\alpha > 0$, $\cJ_\alpha=\cJ_\alpha(\mathcal H)$ is a subspace
  of the compact operators on a Hilbert space $\mathcal H$. For $A \in
  \cJ_\alpha$ we define 
  \begin{equation*}
    \|A\|_{\cJ_\alpha}:=\Big(\sum_{n\in\NN} \mu_n(A)^\alpha\Big)^{1/\alpha}
  \end{equation*}
  to be the  $\ell^\alpha$-quasi-norm of the singular values of $A$.
  Here we denote by $\mu_n(A)$ the singular values of the operator $A$.
  It has the following properties:
\begin{itemize}
\item {\em Quasi-norm property:}
  We have, for $c \in \CC$ and $A,B \in \cJ_\alpha$:
  \begin{equation*}
  \Vert c A \Vert_{\cJ_\alpha} = \vert c \vert \, \Vert A
  \Vert_{\cJ_\alpha}
  \end{equation*}
  and
  \begin{align*}
    \Vert A + B \Vert_{\cJ_\alpha}^\alpha &\le \Vert A
      \Vert_{\cJ_\alpha}^\alpha + \Vert B \Vert_{\cJ_\alpha}^\alpha
    \quad \text{for $\alpha \le 1$},
    \\
    \Vert A + B \Vert_{\cJ_\alpha} &\le
      \Vert A \Vert_{\cJ_\alpha} +
    \Vert B \Vert_{\cJ_\alpha} \quad \text{for $\alpha \ge 1$}.
  \end{align*}
  These inequalities imply that there are constants $C(\alpha,m) > 0$
  such that
  \begin{equation*}
    \Bigl\Vert \sum_{j=1}^m A_j \Bigr\Vert_{\cJ_\alpha} \le
      C(\alpha,m)\sum_{j=1}^m \Vert A_j \Vert_{\cJ_\alpha}.
  \end{equation*}
  For $\alpha \ge 1$ one can choose $C(\alpha,m)=1$.
\item {\em H\"older inequality:} Let $1/\alpha + 1/\beta = 1/\gamma$
  for \emph{any} $\alpha, \beta, \gamma > 0$ and $A \in \cJ_\alpha, B
  \in \cJ_\beta$.  Then $AB \in \cJ_\gamma$ and
  \begin{equation*}
  \Vert A B \Vert_{\cJ_\gamma} \le \Vert A \Vert_{\cJ_\alpha} \, \Vert
  B \Vert_{\cJ_\beta}.
  \end{equation*}
\item {\em Ideal property:} Let $A \in \cJ_\alpha$ and $B$ be a
  bounded operator on the Hilbert space $\mathcal H$. Then we have $A
  B$, $B A \in \cJ_\alpha$ and
  \begin{equation*}
    \Vert A B \Vert_{\cJ_\alpha}\le \Vert A \Vert_{\cJ_\alpha} \, \Vert
    B \Vert, \quad \Vert B A \Vert_{\cJ_\alpha} \le \Vert B \Vert \,
    \Vert A \Vert_{\cJ_\alpha},
  \end{equation*}
  where $\Vert \cdot \Vert$ denotes the usual norm of bounded
  operators on $\mathcal H$.
\item \emph{Monotonicity:} For $\alpha \le \beta$ and $A \in
  \cJ_\alpha$, we have $A \in \cJ_\beta$ and $\Vert A \Vert_{J_\beta}
  \le \Vert A \Vert_{J_\alpha}$.
\end{itemize}

For $p > 1$ let $\alpha=1-1/p$, $q$ be even, $k \in \NN$ and
$g(x)=(1+x)^{-k}$ be given as in Definition~\ref{d-constants}. With
this choice of parameters, the following results hold:
\begin{thm}
  \label{t-deffestP}
  Let $H_\omega$ be a family satisfying RAP. There exists a constant
  $C_\alpha > 0$, such that
  \begin{equation*}
    \Vert g(H^I_{\omega^2}) - g(H^I_{\omega^1})
         \Vert_{\cJ_\alpha}
    \le C_\alpha,
  \end{equation*}
  for all subsets $I \subset \Gamma$ and all $\omega^1$,$ \omega^2 \in
  \Omega$ differing in only one coordinate.
\end{thm}
Note that $\cJ_\alpha = \cJ_\alpha(L^2(X,g_\omega))$ 
and that $g_\omega=g_{\omega^1}=g_{\omega^2}$ since $\omega^1$ and $\omega^2$
differ only in the coupling constant of the potential.
\begin{thm}
  \label{t-deffestM}
  Let $H_\omega$ be a family satisfying RAM.  There exists a constant
  $\hat C_\alpha > 0$, such that
  \begin{equation*}
    \Vert g(H^I_{\omega^2}) - g(S H^I_{\omega^1} S^*)
          \Vert_{\cJ_\alpha}
    \le \hat C_\alpha
  \end{equation*}
  for all subsets $I \subset \Gamma$ and all $\omega^1$, $\omega^2 \in
  \Omega$ differing in only one coordinate and
  $S:=S_{\omega^1,\omega^2}$, defined in~\eqref{e-Somega}.
\end{thm}
Note here, that $\cJ_\alpha = \cJ_\alpha(L^2(X,g_{\omega^2}))$.  The
proofs of the two theorems are similar.  We only present the proof of
Theorem~\ref{t-deffestM}, since it concerns the more complicated case.
Assume that $\omega^1$ and $\omega^2$ differ only in the coordinate
$\gamma \in \Gamma$. For simplicity, set
\begin{equation*}
  H_1 := H_{\omega^1}^I, \quad
  \wt H_1 := S H_{\omega^2}^I S^* \quad\text{and} \quad
  H_2 := H_{\omega^2}^I.
\end{equation*}
Since the single site deformation $u$ is compactly supported, there
exists a radius $R$, such that
\begin{equation*}
  \wt H_1 \phi =
  H_2 \phi \qquad \text{for all} \quad
   \phi \in C_{\mathrm c}^\infty (\Lambda(I) \setminus B_R(\gamma \cF)),
\end{equation*}
where $B_R(\gamma \cF)$ denotes the open $R$-neighborhood of $\gamma
\cF$ with respect to the metric $g_0$. Choose $f_0, F_0 \in \Cci X$
such that
\begin{equation}
  \label{eq:def.f.F}
  f_0\mid_{B_R(\cF)} \equiv 1, \quad \supp f_0 \subset B_{2R}(\cF)
   \qquad \text{and} \qquad F_0\vert_{\supp f_0} \equiv 1,
\end{equation}
Let $f = f_0 \circ \gamma^{-1}$, $F = _0 \circ \gamma^{-1}$ be their
$\gamma$-translates. Then we have
\begin{align*}
  D_{\eff} & :=  g(H_2)- g(\wt H_1) \\ &=
  -\sum_{m=0}^{k-1} (H_2+1)^{-(k-m)} \, (H_2 -\wt H_1) \,
        (\wt H_1+1)^{-(m+1)} \\
  &= -\sum_{m=0}^{k-1} (H_2+1)^{-(k-m)} f\, (H_2 -\wt H_1) \,f
       (\wt H_1+1)^{-(m+1)} \\ &=
  -\sum_{m=0}^{k-1} \big[ f \,
    (H_2+1)^{-(k-m)}\big ]^* \, (H_2-\wt H_1) \, \big [f \, (\wt
     H_1+1)^{-(m+1)}\big].
\end{align*}
Note that all operators in the previous calculation are defined in the
same Hilbert space $L^2(\Lambda(I),g_{\omega^2})$.

By monotonicity, the quasi-norm property and the H\"older inequality,
we obtain
\begin{multline*}
  \Vert D_{\eff} \Vert_{\cJ_\alpha} \le \Vert D_{\eff}
  \Vert_{\cJ_{q/k}}\hfill
  \\\le
  C(q/k,k) \sum_{m=0}^{k-1} \| f \, (H_2+1)^{-(k-m)}
  \|_{\cJ_{q/(k-m)}} \, \|(H_2-\wt H_1) \, f \, (\wt
  H_1+1)^{-(m+1)} \|_{\cJ_{q/m}}.
\end{multline*}
It remains to estimate each of the terms at the right side,
independently of $\omega$, $I$ and $\gamma$.  We explain this for the
most difficult term $\Vert H_2 f (\wt H_1 +1)^{-(m+1)}
\Vert_{\cJ_{q/m}}$. The term $\Vert \wt H_1 f (\wt H_1 +1)^{-(m+1)}
\Vert_{\cJ_{q/m}}$ can be treated similarly, and the term $\Vert f
(H_2+1)^{-(k-m)} \Vert_{\cJ_{q/(k-m)}}$ is even simpler. In each case
we use the following fact, which is in the Euclidean situation
essentially due to Nakamura~\cite{Nakamura-01}:

\begin{prp}
  \label{nakamura}
  Let $\omega \in \Omega$ and $I \subset \Gamma$ be arbitrary.  Let
  $f, F \in \Cci X$ with $F =1$ on $\supp f$, $R = (H_\omega^I
  +1)^{-1}$, and $\nu \in \NN$ be fixed.  Then we have
  \begin{equation*}
  f R^\nu = \sum_{i=1}^{N_\nu} \prod_{j=1}^{\nu} f_{ij} R B_{ij}
    = \sum_{i=1}^{N_\nu} (f_{i1} R B_{i1}) \cdots (f_{i\nu} R B_{i\nu}),
  \end{equation*}
  where $f_{ij} = F$ for $j < \nu$, the functions $f_{i\nu}$ agree
  with certain $\omega$-dependent derivatives of $f$, and the $B_{ij}$
  are bounded operators.

  There exist a constant $\hat C_1(\nu)$, which is independent of
  $\omega$ and $I$, such that $\Vert f_{ij} \Vert_\infty \le \hat
  C_1(\nu)$. The bound $\hat C_1(\nu)$ does not change when replacing
  $f, F$ by any translate $f \circ \gamma^{-1}, F \circ \gamma^{-1}$
  with $\gamma \in \Gamma$.

  Moreover, there exists a constant $\hat C_2 $, which is independent
  of $f, F, \nu$, $\omega$ and $I$, such that $\Vert B_{ij} \Vert \le
  \hat C_2$.
\end{prp}
We will prove this proposition in full detail in Section \ref {s-CommRel}
and describe $f_{ij}$ and $B_{ij}$ explicitely.

Note that all considerations are carried out in the Hilbert space
$L^2(\Lambda(I),g_{\omega^2})$, unless stated otherwise. However,
$\Vert \cdot \Vert_{\cJ_\alpha,\omega^1}$ denotes the
$\cJ_\alpha$-norm with respect to the Hilbert space
$L^2(\Lambda(I),g_{\omega^1})$ and $\Vert \cdot \Vert_{\omega^1}$ is
the corresponding operator norm. Furthermore, we introduce $R_1 :=
(H_1 + 1)^{-1}$.

Let us now return to the study of the term $H_2 f (\wt H_1
+1)^{-(m+1)}$.  The spectral theorem and Proposition~\ref{nakamura}
yield
\begin{eqnarray*}
  H_2 f (\wt H_1 +1)^{-(m+1)}
   &=&
  H_2 S (f R_1^{m+1})S^* \\
   &=&
  \sum_{i=1}^{N_{m+1}} H_2 S
  \Bigl( (f_{i1} R_1 B_{i1})
  \prod_{j=2}^{m+1} (f_{ij} R_1 B_{ij}) \Bigr) S^*\\
  &=&\sum_{i=1}^{N_{m+1}} ( H_2 f_{i1} S R_1 S^* ) \bigg( S B_{i1}
   \Bigl( \prod_{j=2}^{m+1} f_{ij} R_1 B_{ij} \Bigr) S^* \bigg).
\end{eqnarray*}
Using the quasi-norm property, the ideal property and the H\"older
inequality, we obtain
\begin{eqnarray*}
  && \Vert H_2 f (\wt H_1 +1)^{-(m+1)} \Vert_{\cJ_{q/m}}\\
  &\le& C(q/m,N_{m+1}) \sum_{i=1}^{N_{m+1}} \Vert H_2 f_{i1} S R_1 S^*
  \Vert \cdot \Vert S B_{i1} \Bigl( \prod_{j=2}^{m+1} f_{ij} R_1
    B_{ij}\Bigr) S^* \Vert_{\cJ_{q/m}} \\
  &=& C(q/m,N_{m+1}) \sum_{i=1}^{N_{m+1}} \Vert H_2 f_{i1}
  (\wt H_1 + 1)^{-1} \Vert \cdot \Vert B_{i1} \prod_{j=2}^{m+1}
  (f_{ij} R_1 B_{ij}) \Vert_{\cJ_{q/m},\omega^1}\\
  &\le& C(q/m,N_{m+1}) \sum_{i=1}^{N_{m+1}}
  \Vert H_2 f_{i1} (\wt H_1 + 1)^{-1} \Vert \cdot \Vert B_{i1}
  \Vert_{\omega^1} \cdot \prod_{j=2}^{m+1} \Vert f_{ij} R_1
  B_{ij} \Vert_{\cJ_q,\omega^1}.
\end{eqnarray*}
Note that, by the ideal property of the spaces $\cJ_q$, we have
\begin{equation*}
  \Vert f_{ij} R_1 \Vert_{\cJ_q,\omega^1} \le \hat C_1(\nu) \Vert F R_1
  \Vert_{\cJ_q,\omega^1}
\end{equation*}
due to Proposition~\ref{nakamura}, since the support of any derivative
of $f$ is contained in the support of $f$. As this proposition also
gives $\Vert B_{ij} \Vert_{\omega^1} \le \hat C_2$, we continue our
estimate as follows:
\begin{multline}\label{eq:cc2c1}
   \Vert H_2 f (\wt H_1 +1)^{-(m+1)} \Vert_{\cJ_{q/m}}\hfill\\
  \le C(q/m,N_{m+1}) (\hat C_2)^{m+1}
  \sum_{i=1}^{N_{m+1}} \Vert H_2 f_{i1} (\wt H_1 + 1)^{-1}
  \Vert \cdot \prod_{j=2}^{m+1} \Vert f_{ij} R_1
  \Vert_{\cJ_q,\omega^1} \hfill \\
  \le C(q/m,N_{m+1})  (\hat C_2)^{m+1} \hat C_1(m+1)^m 
     \sum_{i=1}^{N_{m+1}}
  \Vert H_2 f_{i1} (\wt H_1 + 1)^{-1} \Vert  \Vert F R_1
  \Vert_{\cJ_q,\omega^1}^m.\hfill
\end{multline}

Note that $f_{i1}$ in the above formula agrees with $f$ or $F$. Thus
the left hand factor in the last sum above can be estimated by the
following lemma:

\begin{lem}
  \label{hilf1}
  There exists a constant $C_0 > 0$, independent of $\omega^1$,
  $\omega^2$, $I$, and $\gamma$
  such that
  \begin{equation*}
  \Vert H_2 f (\wt H_1 + 1)^{-1} \Vert, \Vert H_2 F (\wt H_1 +
  1)^{-1} \Vert \le C_0.
  \end{equation*}
  where $\|\cdot \|$ is the norm of bounded operators in
  $L^2(\Lambda(I),g_{\omega^2})$.
\end{lem}

Note that $\gamma$ enters into the definition of $f$ and $F$,
see~\eqref{eq:def.f.F}.

\begin{proof}
  We use the notation of the appendix. Due to the relative boundedness
  of the family $\{g_\omega\}_\omega$ with respect to the periodic
  metric $g_0$, the Sobolev spaces $W^k(\Lambda(I),g_\omega)$ and
  $W^k(\Lambda(I),g_0)$ are equivalent with constants
  \emph{independent} of $\omega$. We do not mention these
  identifications in the rest of the proof.

  By Lemma~\ref{lem:mult.op}, the operators, given by the
  multiplication with the smooth functions $f=f_0 \circ \gamma^{-1}$
  and $F = F_0 \circ \gamma^{-1}$, are bounded in
  $W^2(\Lambda(I),g_0)$ with constants obviously independent of $I$.
  The independence of $\gamma$ for $(\Lambda(I),g_0)$ follows by
  periodicity of the metric~$g_0$.

  Moreover, using Theorem~\ref{thm:ell.reg} and the uniform
  infinitesimal boundedness of the potential (see Remark~\ref{r-pot}),
  the identification operators
  $$
    {\rm Id}_1 \colon W^2(\Lambda(I),H_1) 
               \to W^2(\Lambda(I),g_{\omega^1})
       \quad \text{and}\:\; 
    {\rm Id}_2 \colon  W^2(\Lambda(I), g_{\omega^2}) 
              \to W^2(\Lambda(I),H_2) $$
  are also bounded uniformly in $I$ and $\omega$.  Recall the
  definition of the multiplication operator $S=S_{\omega^1,\omega^2}$
  in~\eqref{e-Somega} or~\eqref{S}. It follows from
  Lemma~\ref{lem:mult.op}, that $S$ acting on $W^2(\Lambda(I),g_0)$ (resp.\
  $S^*$ acting on $L^2(\Lambda(I),g_0)$) is uniformly bounded in $\omega$ (and
  in $I$) by the relatively boundedness of $\{g_\omega\}_\omega$.
  
  Finally, $R_1\colon L^2(\Lambda(I),g_{\omega^1}) \to
  W^2(\Lambda(I),H_1)$ is an isometry and
  \begin{equation*}
    H_2\colon W^2(\Lambda(I),H_2) \to L^2(\Lambda(I),g_{\omega^2})
  \end{equation*}
  is bounded in norm by $1$.  The statement of the lemma follows now
  by writing the two operators as the compositions $H_2 \, \Id_2 \, f
  S \, \Id_1 R_1 S^*$ and $H_2 \, \Id_2 \, F S \, \Id_1 R_1 S^*$ of
  uniformly bounded operators (and the hidden identification of the
  spaces depending on $g_{\omega^1}$, $g_{\omega^2}$, and $g_0$).
\end{proof}

For the remaining terms in~\eqref{eq:cc2c1} we use the following lemma:

\begin{lem}
  \label{hilf2}
  There is a constant $C_1 > 0$, independent of $\omega$, $I$,
  $\gamma$, such that
  \begin{equation*}
    \Vert F R_1 \Vert_{\cJ_{q},\omega^1} \le C_1.
  \end{equation*}
\end{lem}

The proof of Lemma~\ref{hilf2} is somewhat involved and is presented
in Section~\ref{s-TraceRes}.  By Lemmas~\ref{hilf1} and~\ref{hilf2},
we finally obtain the estimate
$$
  \Vert H_2 f (\wt H_1 +1)^{-(m+1)} \Vert_{\cJ_{q/m}} \\
  \le N_{m+1}
    C(q/m, N_{m+1}) C_0 \hat C_2 (\hat C_1(m+1)  (C_1) \hat C_2) )^{m}.$$
Note that all constants are independent of $\omega, I$ and $\gamma$.
This completes the proof of the uniform boundedness of $\Vert g(H_2) -
g(\wt H_1)\Vert_{\cJ_\alpha}$ up to the proofs of
Proposition~\ref{nakamura} and Lemma~\ref{hilf2}. These proofs are
given in the next two sections.


\section{Commutator relations and estimates}
\label{s-CommRel}

This section is devoted to the proof of Theorem~\ref{t-nakamura}
below.  It implies Proposition~\ref{nakamura} and, moreover, provides
an explicit description of the operators $f_{ij}$ and $B_{ij}$.
Roughly, we want to rewrite $f R^\nu$ as as product of $\nu$ factors
of the type $f_{ij} R B_{ij}$.  The key idea is to use a certain commutator
relation iteratively, similarly as in~\cite{Nakamura-01}.

To clarify some formulae in this section we will occasionally use the
notation $M_f$ for the multiplication operator by $f$.  Let $\omega
\in \Omega$ and $I \subset \Gamma$ be arbitrary. For simplicity, we
drop the dependency on $\omega$, $I$ in this section and write
$\grad$, $\Delta$, $\ddiv$, $V$ for $\grad_\omega$, $\Delta_\omega$,
$\ddiv_\omega$, $V_\omega$ and $H$ for $H_\omega^I$. Only for the
metric we keep the notation $g_\omega$ to distinguish it from the
periodic metric $g_0$.  Recall that we use the convention $\Delta =
\ddiv \grad \le 0$.  Moreover, let $R := (H+1)^{-1}$.

\begin{lem}[Commutator lemma] \label{commut1}
  For any function $h \in C^\infty_{\mathrm c}(X)$ we have
  \begin{equation}
    \label{comm}
    h R = R h - R h^{\{1\}} R - R \ddiv h^{\{2\}} R,
  \end{equation}
  where $h^{\{1\}} = \Delta h, h^{\{2\}} = -2 \grad h$.
\end{lem}

\begin{proof}
  We first prove
  \begin{equation} \label{cDM}
    [-\Delta, M_h] = M_{h^{\{1\}}} + \ddiv M_{h^{\{2\}}}.
  \end{equation}
  This follows from
  \begin{eqnarray*}
    [-\Delta, M_h] \phi & =& -\Delta( h \phi) + h \Delta \phi\\
    &=& -\ddiv(\phi \grad h + h \grad \phi) + \ddiv(h \grad \phi) -
    g_\omega(\grad h, \grad \phi)\\ 
&=& - \ddiv(\phi \grad h) -
    g_\omega(\grad h,\grad \phi)\\
& =& -2\ddiv(\phi \grad h) +
    \phi \ddiv(\grad h)\\ 
&=& \ddiv (M_{h^{\{2\}}}\phi) + M_{h^{\{1\}}} \phi.
  \end{eqnarray*}
  {}From the resolvent equation we obtain
  \begin{equation*}
  [M_h,R] = R(-\Delta M_h + M_h \Delta)R = R [-\Delta, M_h] R,
  \end{equation*}
  and, using \eqref{cDM}, we conclude that
  \begin{equation*}
  [M_h,R] = R (M_{h^{\{1\}}} + \ddiv M_{h^{\{2\}}}) R = R M_{h^{\{1\}}} R + R \ddiv
  M_{h^{\{2\}}} R,
  \end{equation*}
  which proves the lemma.
\end{proof}

A key idea is to apply the above lemma, a second time, to the
expression $h^{\{2\}} R$ in \eqref{comm}. However, $h^{\{2\}}$ is a
vector field.  We solve this problem by introducing the operators
$\ddiv_{i,\beta}$, acting on functions, in the following way: Let
$(\psi_\beta)_{\beta \in B}$ be a finite partition of unity on the
compact manifold $M$, i.e.,
\begin{equation*}
  \sum_{\beta=1}^n \psi_\beta = 1.
\end{equation*}
Moreover, for all $\beta$, let $X_{1,\beta}, \dots, X_{d,\beta}$ be
vector fields which are a local orthonormal frame on the subset $\supp
\psi_\beta \subset M$ with respect to the metric $g_0$.

Let $\pi \colon X \to M$ be the canonical projection and let us denote the
periodic lifts $\psi_\beta \circ \pi$ and $X_{i,\beta} \circ D\pi$ on
$X$, again, by $\psi_\beta$ and $X_{i,\beta}$, for simplicity. Note
that every vector field $Z \in C^\infty(TX)$ can be written as
\begin{equation*}
   Z = \sum_{i,\beta} \psi_\beta\, g_0(Z, X_{i,\beta}) X_{i,\beta},
\end{equation*}
We define the operator $\ddiv_{i,\beta}$ by
\begin{equation*}
  \ddiv_{i,\beta}(h) := \ddiv(\psi_\beta h X_{i,\beta}) = g_\omega(
  \psi_\beta X_{i,\beta}, \grad h ) + h \ddiv(\psi_\beta X_{i,\beta})
\end{equation*}
and obtain
\begin{equation} \label{ddiv}
  \ddiv Z = \sum_{i,\beta} \ddiv_{i,\beta}(g_0(Z, X_{i,\beta})).
\end{equation}
Note that the operator $\ddiv$ and therefore also $\ddiv_{i,\beta}$ is
$\omega$-dependent, since $\ddiv$ is defined via the metric
$g_\omega$.

Using the differential operators $\ddiv_{i,\beta}$, we can reformulate
the above lemma in the following way:

\begin{cor} \label{commut2}
  For any function $h \in C^\infty_{\mathrm c}(X)$ we have
  \begin{equation*}
  h R = R h + R h^{\{1\}} R + \sum_{i,\beta} R \,\ddiv_{i,\beta}\,
  h^{2,i,\beta} R,
  \end{equation*}
  where $h^{\{1\}} = \Delta h, h^{2,i,\beta} = -2\, g_0( \grad h,
  X_{i,\beta})$.
\end{cor}

Now, we can apply Corollary \ref{commut2} twice and obtain the
following result, which is of central importance. In
formula~\eqref{e-commut3} below we use the convention that expressions
of the form $(D h)$ denote multiplication operators by the function
$Dh$.

\begin{prp}
  \label{commut3}
  For any function $h \in \Cci X$ we have
  \begin{multline}
    \label{e-commut3}
    h R = R h + R (D^{(1)} h) R + R \ddiv_{i,\beta} R (D^{(2,i,\beta)} h)\\
    + R \ddiv_{i,\beta} R (D^{(3,i,\beta)} h) R + R \ddiv_{i,\beta} R
    \ddiv_{j,\mu} (D^{(4,i,\beta,j,\mu)} h) R,
  \end{multline}
  where $D^{(1)} h = \Delta h$, $D^{(2,i,\beta)} h = -2\, g_0( \grad
  h, X_{i,\beta} )$, $D^{(3,i,\beta)} h = D^{(1)} D^{(2,i,\beta)} h$
  and $D^{(4,i,\beta,j,\mu)} h = D^{(2,j,\mu)} D^{(2,i,\beta)} h$ are
  compactly supported function with support contained in $\supp h$.
\end{prp}
Note that we used Einstein notation and omitted sum signs, for
simplicity.
\begin{proof}
 A first application of Corollary \ref{commut2} gives
$$    h R = R h + R (D^{(1)}h) R + R \ddiv_{i,\beta}
    \bigl((D^{(2,i,\beta)}h) R \bigr).$$
We now apply  Corollary \ref{commut2} again to  the term $(D^{(2,i,\beta)}h) R$ and obtain the desired statement.
\end{proof}

Now, we can formulate a more detailed version of
Proposition~\ref{nakamura}:

\begin{thm}\label{t-nakamura}
  Let $f, F \in \Cci X$ with $F =1$ on $\supp f$ and $\nu \in
  \NN$ be fixed.  Then we have
  \begin{equation} 
    \label{e-fRnu}
    f R^\nu = \sum_{i=1}^{N_\nu} \prod_{j=1}^{\nu} f_{ij} R B_{ij}
    = \sum_{i=1}^{N_\nu} (f_{i1} R B_{i1}) \cdots (f_{i\nu} R B_{i\nu}).
  \end{equation}
  Here, $f_{ij} = F$ for $j < \nu$, and the functions $f_{i\nu}$ are
  of the form $D f$, where $D$ is a composition of $\nu-1$ operators
  of the set ${\mathcal D} := \{ {\rm Id},
  D^{(1)},D^{(2,i,\beta)},D^{(3,i,\beta)},D^{(4,i,\beta,j,\mu)}\}$.
  Morover, the operators $B_{ij}$ are bounded and of the form $B R^l$
  with $B \in {\mathcal B} := \{ \Id, R, \ddiv_{i,\beta} R,
  \ddiv_{i,\beta} R \ddiv_{j,\mu} \}$ and $0 \le l \le \nu-1$.

  There is a constant $\hat C_1(\nu)$, which does not depend on
  $\omega \in \Omega$ and $I \subset \Gamma$ such that
  \begin{equation}
    \label{fijest}
    \Vert f_{ij} \Vert_\infty \le \hat C_1(\nu).
  \end{equation}
  The bound $\hat C_1(\nu)$ does not change when replacing $f, F$ by
  any translate $f \circ \gamma^{-1}, F \circ \gamma^{-1}$ with
  $\gamma \in \Gamma$.

  Finally, there is a constant $\hat C_2$, which does not depend on
  $\nu\in \NN$, $\omega \in \Omega$, $I \subset \Gamma$, and $f,F\in
  \Cci X$ such that
  \begin{equation} \label{Bijest}
    \Vert B_{ij} \Vert \le \hat C_2.
  \end{equation}
\end{thm}

The proof of this theorem needs some preparation and will be given at
the end of this section.
\begin{lem} 
  \label{l-Dest}
  Let $f \in \Cci X$ and $\nu \in \NN$ be fixed. Then there exists a
  constant $C_1(\nu) > 0$, independent of $\omega$ such that
  \begin{equation*}
     \Vert D (f\circ \gamma^{-1}) \Vert_\infty \le C_1(\nu),
  \end{equation*}
  \sloppy
  for all $\gamma \in \Gamma$ and every composition $D$ of $2\nu-2$
  operators of the set $\{ {\rm Id}, D^{(1)}, D^{(2,i,\beta)} \}$,
  where $D^{(1)} f = \Delta_\omega f$ and $D^{(2,i,\beta)} f = -2\,
  g_0(\grad_\omega f, X_{i,\beta})$.
\end{lem}

\begin{proof}
  The dependence of $\gamma$ can easily be eliminated by the
  observation, that all operators $D$ satisfy the equivariance
  condition~\eqref{compcomp}. For example we have
  $D_\omega^{(2,i,\beta)} (f \circ \gamma^{-1}) = (D_{\gamma^{-1}
    \omega}^{(2,i,\beta)} f) \circ \gamma^{-1}$ because of
  \begin{eqnarray*}
    D_\omega^{(2,i,\beta)} U_{\omega,\gamma} f & = & -2
    g_0(\grad_\omega (f \circ \gamma^{-1}), X_{i,\beta}) = -2(
    A_\omega^{-1} X_{i,\beta}) (f \circ \gamma^{-1}) \\ & =& -
    2((A_{\gamma^{-1}\omega}^{-1} X_{i,\beta})f)\circ \gamma^{-1} =
    U_{\omega,\gamma} D_{\gamma^{-1}\omega}^{(2,i,\beta)} f,
  \end{eqnarray*}
  where we used the $\Gamma$-periodicity of $X_{i,\beta}$. Therefore,
  the dependence of $\gamma$ can be moved into a dependence of
  $\omega$. The supremum norm estimates follow easily
  from the observation, that the operators $D$ depend only on
  $g_\omega$, its derivatives and on $X_{i,\beta}$, which are bounded
  in a suitable atlas (see~(v') of Definition~\ref{def:bdd.geo}).
\end{proof}

\begin{lem} 
  \label{l-dRdest}
  The operators $R$, $R\ddiv$, $\grad R$ and $\grad R \ddiv$ on
  $L^2(\Lambda(I),g_\omega)$ are bounded operators with norm $\le 1$.
\end{lem}

\begin{proof}
  By the spectral theorem, both $R$ and $R^{1/2}$ are bounded by one.
  Next we prove boundedness of $R^{1/2}\ddiv$. For the proof we use
  the differential form calculus. $\Vert R^{1/2}\ddiv \Vert \le 1$
  translates then into the condition
  \begin{equation*}
   \langle R d^* \eta, d^* \eta \rangle \le
     \Vert \eta \Vert^2
  \end{equation*}
  for all one-forms $\eta \in \Omega^1_{\mathrm c}(\Lambda(I))$ with
  compact support. Since inversion is a monotone operator function and
  $-\Delta \le H=-\Delta+V$, we conclude that $R = (H +1)^{-1} \le
  (-\Delta + 1)^{-1}$, so it remains to prove
  \begin{equation*}
    \langle d (-\Delta +1)^{-1} d^* \eta, \eta \rangle \le
    \Vert \eta \Vert^2 \qquad \text{for all $\eta \in
    \Omega^1_{\mathrm c}(\Lambda(I))$}.
  \end{equation*}
  Adding non-negative terms and using $-\Delta = (d+d^*)^2$, it
  suffices to prove that
  \begin{equation*}
    \langle (d+d^*) (-\Delta +1)^{-1} (d+d^*) \eta,
      \eta \rangle \le
    \Vert \eta \Vert^2 \qquad \text{for all $\eta
     \in \Omega^1_{\mathrm c}(\Lambda(I))$},
  \end{equation*}
  which follows from the spectral theorem applied to the elliptic
  operator $d+d^* \colon \Omega_{\mathrm c}(\Lambda(I)) \to
  \Omega_{\mathrm c}(\Lambda(I))$.

  The formal adjoint of $R^{1/2}\ddiv$ is $-\grad R^{1/2}$, so we
  conclude $\Vert \grad R^{1/2} \Vert \le 1$, and finally $\Vert \grad
  R \ddiv \Vert \le 1$, by composition.
\end{proof}

\begin{lem} \label{l-diffest}
  There is a constant $\hat C_2 > 0$, which does not depend on $\omega
  \in \Omega$ and $I \subset \Gamma$, such that
  \begin{equation*}
  \Vert R \Vert,\ \Vert \ddiv_{i,\beta} R \Vert,\ \Vert
  \ddiv_{i,\beta} R \ddiv_{j,\mu} \Vert \le \hat C_2.
  \end{equation*}
\end{lem}

\begin{proof}
  Since by definition
  \begin{equation*}
  \ddiv_{i,\beta} R \phi = (R \phi) \ddiv(\psi_\beta X_{i,\beta}) +
  g_\omega(\grad R \phi, \psi_\beta X_{i,\beta} ),
  \end{equation*}
  we conclude with Lemma \ref{l-dRdest} that
  \begin{eqnarray*}
    \Vert \ddiv_{i,\beta}R\phi \Vert & \le & \Vert \ddiv(\psi_\beta
    X_{i,\beta}) \Vert_\infty \cdot \Vert R\phi \Vert + \Vert
    \psi_\beta X_{i,\beta} \Vert_{\infty,g_\omega} \cdot \Vert \grad R
    \phi \Vert \\ & \le & \bigl( \Vert \ddiv(\psi_\beta X_{i,\beta})
      \Vert_\infty + C_{{\rm rel},0}^{1/2} \Vert \psi_\beta X_{i,\beta}
      \Vert_{\infty,g_0} \bigr) \Vert \phi \Vert,
  \end{eqnarray*}
  where $C_{{\rm rel},0}$ is the uniform quasi-isometry constant
  in~\eqref{quasiisom} comparing the metrics $g_0$ and $g_\omega$.
  Note that $X_{i,\beta}, \psi_\beta$ are periodic and independent of
  the choices $\omega, I$ and that the term $\Vert
  \ddiv_\omega(\psi_\beta X_{i,\beta}) \Vert_\infty$ can be uniformly
  bounded for all $\omega \in \Omega$ by the relative boundedness
  assumptions on the metrics $g_\omega$.

  Similarly,
  \begin{eqnarray*}
 & &  \ddiv_{i,\beta} R \ddiv_{j,\mu} \phi\\ 
& = &\ddiv \left( (R\ddiv
      M_{\psi_\mu X_{j,\mu}} \phi)
      \psi_\beta X_{i,\beta} \right) \\
& = &g_\omega( \psi_\beta X_{i,\beta},
               \grad R \ddiv M_{\psi_\mu X_{j,\mu}} \phi ) +
      M_{\ddiv(\psi_\beta X_{i,\beta})}
            R \ddiv M_{\psi_\mu X_{j,\mu}} \phi
  \end{eqnarray*}
  implies that
  \begin{multline*}
    \Vert \ddiv_{i,\beta} R \ddiv_{j,\mu} \phi \Vert \\\le
      \bigl(
        C_{{\rm rel},0}^{1/2} \Vert \psi_\beta X_{i,\beta}\Vert_{\infty,g_0}
      \, \Vert \grad R \ddiv \Vert + \Vert \ddiv(\psi_\beta
      X_{i,\beta}) \Vert_\infty \, \Vert R \ddiv \Vert \bigr) \cdot\\
    \cdot C_{{\rm rel},0}^{1/2}\, \Vert \psi_\mu X_{j,\mu}
    \Vert_{\infty,g_0}\, \Vert \phi \Vert\\
    \le \bigl( C_{{\rm rel},0}^{1/2} \Vert \psi_\beta
      X_{i,\beta}\Vert_{\infty,g_0} + \Vert \ddiv(\psi_\beta
      X_{i,\beta}) \Vert_\infty \bigr) C_{{\rm rel},0}^{1/2}\, \Vert
    \psi_\mu X_{j,\mu} \Vert_{\infty,g_0}\, \Vert \phi \Vert.
  \end{multline*}
\end{proof}

\begin{proof}[Proof of Theorem \ref{t-nakamura}]
  We first prove the commutator relation~\eqref{e-fRnu} by induction.
  The equation is obviously satisfied in the case $\nu =1$ with $N_1 =
  1$, $f_{11} = f$, $B_{11} = \Id$. Assume that the equation is true
  for $\nu-1$. Using Proposition~\ref{commut3} we obtain
  \begin{align*}
    f R^\nu &=  F (f R) R^{\nu -1}\\
    & = (F R) (f R^{\nu-1})
       + (F R)((D^{(1)}f) R^{\nu-1}) R \\
       &\quad+ (F R (\ddiv_{i,\beta} R))((D^{(2,i,\beta)}f) R^{\nu-1})\\
       &\quad+ (F R (\ddiv_{i,\beta} R))
             ((D^{(3,i,\beta)}f) R^{\nu-1})R \\
       &\quad+ (F R (\ddiv_{i,\beta} R
              \ddiv_{j,\mu})) ((D^{(4,i,\beta,j,\mu)}f) R^{\nu-1})R.
  \end{align*}
  Note that each term involved is of the form $(F R B)((\wt D f)
  R^{\nu-1}) R^s$ with $B \in {\mathcal B}$, $s \in \{0,1\}$ and $\wt
  D \in {\mathcal D}$. Using the induction hypothesis we conclude that
  \begin{equation*}
    (\wt D f) R^{\nu-1} = \sum_{i=1}^{N_{\nu-1}} \prod_{j=1}^{\nu-1}
  g_{ij} R B_{ij},
  \end{equation*}
  where $g_{ij}$ is of the form $D \wt D f$ and $D$ is a composition of
  $\nu-2$ operators in ${\mathcal D}$, and the operators $B_{ij}$ are
  of the form $B R^l$ with $B \in {\mathcal B}$ and $0 \le l \le
  \nu-2$.  This finishes the induction step.

  The norm estimates~\eqref{fijest} and~\eqref{Bijest} are easy
  consequences of Lemmas~\ref{l-Dest} and~\ref{l-diffest}.
\end{proof}


\section{A trace class estimate of the resolvent}
\label{s-TraceRes}

In this final section we prove the following proposition:
\begin{prp}
  \label{FR}
  Let $F_0 \in \Cci X$ be a fixed smooth function with compact
  support. For $I \subset \Gamma$ and $\omega \in \Omega$, let
  $R^I_\omega := (H^I_\omega + 1)^{-1}$.  Then there is a constant $C
  > 0$, independent of $\omega$ and $I$ such that
  \begin{equation*}
    \Vert F_0\, R^I_\omega \Vert_{\cJ_q,\omega} \le C.
  \end{equation*}
\end{prp}

Recall that Lemma~\ref{hilf2} claims $\Vert (F_0 \circ \gamma)
R^I_\omega \Vert_{\cJ_{q,\omega}} \le C$, independently of the choice
of $\omega, I$ and $\gamma \in \Gamma$. This, however, is an immediate
consequence of Proposition \ref{FR} and the equivariance property
\eqref{compcomp} of the operators $H_\omega$: Using the unitary map
$U_{(\gamma \omega,\gamma)}\colon L^2(\Lambda(\gamma I),g_{\gamma
  \omega}) \to L^2(\Lambda(I),g_\omega)$ for a given $\gamma \in
\Gamma$, we conclude that
\begin{equation*}
   U_{(\gamma \omega,\gamma)}\, (F_0 \circ \gamma)\, R^I_\omega\,
             U_{(\gamma \omega,\gamma)}^* 
   = F_0\, R^{\gamma I}_{\gamma \omega},
\end{equation*}
and therefore
\begin{equation*}
  \Vert (F_0 \circ \gamma) R^I_\omega \Vert_{\cJ_q,\omega} 
   = \Vert F_0 R^{\gamma I}_{\gamma \omega} \Vert_{\cJ_q,\gamma \omega}.
\end{equation*}
Hence, Proposition~\ref{FR} implies Lemma~\ref{hilf2} and we are left
with the proof of the proposition.
\begin{proof}
  The proof of Proposition \ref{FR} is carried out in three steps:
  \subsubsection*{First Step: Removal of the potential and the
    $\Lambda(I)$-restriction:}
  Let $R^I_{0,\omega} := (- \Delta^I_\omega + 1)^{-1}$. Using the
  ideal property, we obtain
  \begin{equation*}
     \Vert F_0 R^I_\omega \Vert_{\cJ_q,\omega} 
   = \Vert F_0 R^I_{0,\omega} \Vert_{\cJ_q,\omega} \, 
               \Vert (- \Delta^I_\omega + 1) 
         R^I_\omega \Vert_{L^2(\Lambda(I),g_\omega)} 
   \le C_0 \Vert F_0 R^I_{0,\omega} \Vert_{\cJ_q,\omega}.
  \end{equation*}
  Note that the constant $C_0 > 0$ is independent of $\omega$ and $I$,
  because of the uniform infinitesimal $\Delta_\omega$-boundedness of
  the potential $V_\omega$ (see Remark \ref{r-pot}). Therefore if
  suffices to estimate the potential free case.

  From the appendix, we infer that the constants of bounded geometry
  of the manifolds $(X,g_\omega)$ and $(\Lambda(I),g_\omega)$ can be
  chosen independently of $\omega$ and $I$ (see
  Lemmas~\ref{lem:rel.bd} and~\ref{lem:rel.bd.aggl}). Set $R_{0,\omega}
  =( -\Delta_\omega + 1)^{-1}$.  Using the ideal property, we obtain
  \begin{align*}
    \Vert F_0 R^I_{0,\omega} \Vert_{\cJ_q,\omega} &
    \le \Vert F_0 (-\Delta_\omega +1)^{-1} 
          \Vert_{\cJ_{q,\omega}} \, 
          \Vert (-\Delta_\omega +1) \mathcal{E} 
          (-\Delta^I_\omega +1)^{-1} \Vert\\&
    =  \Vert F_0 R_{0,\omega} \Vert_{\cJ_{q,\omega}} \, \Vert
    (-\Delta_\omega +1) \mathcal{E} (-\Delta^I_\omega +1)^{-1} \Vert,
  \end{align*}
  where the first norm at the right side is a (super-)trace norm of
  $L^2(X,g_\omega)$ and the second is the operator norm. Here,
  $\mathcal E$ is the extension operator from $W^2(\Lambda(I),\cA)$
  into $W^2(X,\wt \cA)$ as given in Theorem~\ref{thm:ext.op}.  

  From the equivalence of the Sobolev norms (see
  Lemma~\ref{lem:eq.loc.glob} and Theorem~\ref{thm:ell.reg}) and the
  Sobolev extension Theorem~\ref{thm:ext.op}, we conclude that there
  is another constant $C_1 > 0$ (independent of $I$ and $\omega$) such
  that
  \begin{equation*}
    \Vert F_0 R^I_{0,\omega} \Vert_{\cJ_q,\omega} 
    \le C_1 \Vert F_0 R_{0,\omega} \Vert_{\cJ_{q,\omega}}.
  \end{equation*}
  It remains to prove $F_0 R_{0,\omega} \in \cJ_q(L^2(X,g_\omega))$
  and to derive a uniform estimate for the (super-)trace class norm.

  \subsubsection*{Second Step: Hilbert-Schmidt norm estimate for $F_0
    (R_{0,\omega})^{q/2}$:} 
  Note that $K := {\rm supp}\, F_0 \subset X$ is a compact set.  We
  first convince ourselves that
  \begin{equation*}
     F_0 (-\Delta_\omega +1)^{-q/2} \colon
           L^2(X,g_\omega) \to W^q(X,-\Delta_\omega)
  \end{equation*}
  is bounded: $(-\Delta_\omega +1)^{-q/2}\colon L^2(X,g_\omega) \to
  W^q(X,-\Delta_\omega)$ is by definition norm-preserving. By
  Lemma~\ref{lem:mult.op}, the multiplication with $F_0$ is a bounded
  operator in $W^q(X,g_\omega)$ with norm bounded by a constant $C_2$
  depending only on $q$ and $d$, and pointwise bounds on
  $|\nabla_\omega^i F_0|_\omega$, $i=0, \dots, q$. 
  But the latter can
  be estimated by $\omega$-independent constants using the constants
  $C_{\mathrm{rel},i}$ of Definition~\ref{def:rel.bdd} and bounds on
  $|\nabla_0^i F_0|_0$.

  By the Sobolev embedding Theorem \ref{thm:sob.emb}, the identity map
  $W^q(X,-\Delta_\omega) \to \Cb X$ is bounded and its norm can be
  estimated by geometric constants which hold uniformly for all
  manifolds $(X,g_\omega)$; note that $q/2 \ge d/4 + 1$ by
  Definition~\ref{d-constants}.  Consequently,
  \begin{equation*}
    F_0 (R_{0,\omega})^{q/2} \colon 
         L^2(X,g_\omega) \to \Cb K
  \end{equation*}
  is a bounded operator with norm bounded by a constant $C_3 > 0$,
  depending only on $F_0$ and uniform $\omega$-independent geometric
  constants. Now we can apply Theorem \ref{t-HSkernel} and obtain that
  $F_0 (R_{0,\omega})^{q/2}$ is Hilbert-Schmidt with norm bounded by
  $C_3\, (\vol_\omega K)^{1/2} \le C_3'\, (\vol_0 K)^{1/2}$.

  \subsubsection*{Final Step: Trace class estimate for $F_0
    R_{0,\omega}$:}
  Using Lemma 2 of \cite{Brasche-01} (where $J$ equals the
  multiplication operator by $F_0$, $r=0$, $t=1$, $u=q/2$, $p=q$,
  $\alpha = 1$ and $G_\alpha = R_{0,\omega}$), we conclude from the
  second step that $F_0 R_{0,\omega} \in \cJ_q((L^2(X,g_\omega))$ and
  \begin{equation*}
    \Vert F_0 R_{0,\omega} \Vert_{\cJ_q,\omega}^q 
    \le \Vert F_0 \Vert_\infty^{q-2} \, 
    \Vert F_0 (R_{0,\omega})^{q/2} \Vert_{\cJ_2,\omega}^2
    \le (C_3')^2\, \vol_0({\rm supp}\, F_0)\, \Vert F_0 \Vert_\infty^{q-2}
  \end{equation*}
  and were are done.
\end{proof}

\appendix


\section{}
\label{ss-Olaf}

In this appendix, we first define several Sobolev spaces and show that
they are equivalent under certain geometric assumptions. Most of the
material is standard (see e.g.~\cite{Eichhorn-88,Schick-01}).
Afterwards we prove an extension theorem and a Sobolev embedding
theorem.  Note, that it is crucial for our applications, that the
involved constants are \emph{independent} of the random parameter
$\omega$ in the random metric family $\{g_\omega\}_\omega$ and the
choice of $I \subset \Gamma$ in the agglomerates $\Lambda(I)$.
Finally, we recall a Hilbert-Schmidt norm estimate for operators
with continuous kernels.


\subsection{Sobolev spaces on manifolds}
\label{sec:sob.spaces}

Suppose that $M$ is a manifold (possibly with boundary). Suppose, in
addition, that $\{\phi_\alpha\}_\alpha$ is an atlas of $M$ with charts
$\phi_\alpha \colon V_\alpha \to U_\alpha$, where $U_\alpha$ is an
open cover of $M$ and $V_\alpha \subset [0,\infty {[} \times
\RR^{d-1}$. Let $\{\chi_\alpha \}_\alpha$ be a subordinated family of
smooth functions satisfying $\sum\nolimits_\alpha \chi_\alpha^2 = 1$.
Note that $\{\chi_\alpha^2 \}_\alpha$ forms a \emph{partition of
  unity}. We refer to the pair of families $\cA := \{\phi_\alpha,
\chi_\alpha \}_\alpha$ as an \emph{atlas}.

Now, we will define three different types of Sobolev spaces.  The
\emph{local Sobolev space} $W^k(M,\cA)$ of order $k$ with respect to
the atlas $\cA$ is given as the space of function with finite norm
\begin{equation}
  \label{eq:sob.loc}
  \| u \|^2_{W^k(M,\cA)} :=
    \sum_{\alpha \in A} \| \chi_\alpha u_\alpha\|^2_{W^k(V_\alpha)}
\end{equation}
where $u_\alpha := u \circ \phi_\alpha$ and the norm on the RHS is the
usual Sobolev norm in $\RR^d$.

Associated with a Riemannian metric $g$ on $M$, we define the
\emph{global Sobolev space} $W^k(M,g)$ as the space of function with
finite norm
\begin{equation}
  \label{eq:sob.glob}
  \| u \|^2_{W^k(M,g)} :=
    \sum_{i=0}^k \| |\nabla_g^i u|_g \|^2_{L^2(X,g)}
\end{equation}
where $|\nabla_g^i u|_g$ is the pointwise norm of the $i$th covariant
derivative (in the weak sense) with respect to the metric $g$.

Finally, associated with a non-negative (self-adjoint) operator $H$ on
$M$ (usually $H=-\Delta_M$ or $H=-\Delta_M+V$) we define the
\emph{graph norm Sobolev space} with respect to the operator $H$ as
$W^k(M,H) :=\dom (H+1)^{k/2}$ with norm
\begin{equation}
  \label{eq:sob.graph}
  \| u \|_{W^k(M,H)} := \| (H+1)^{k/2} u \|_{L^2(X,g)}.
\end{equation}


\subsection{Manifolds of bounded geometry}
\label{sec:bdd.geo}

In the following we provide the general geometric setting for which we
will establish our results on Sobolev spaces.  We adopt the notion
of~\cite[Sec.~3]{Schick-96} or~\cite{Schick-01}.  Denote by $B_M(x,r)$
the open ball of radius $r$ around $x$ in $(M,g)$.
\begin{dfn}
\label{def:bdd.geo}
A Riemannian manifold $(M,g)$ with boundary $\bd M$ is of
\emph{bounded geometry} iff the following conditions are fulfilled for
constants $r_0>0$, and $C_k>0$ for $k \in \NN, k \ge 0$:
\renewcommand{\labelenumi}{(\roman{enumi})}
  \begin{enumerate}
  \item The collar map
    \begin{equation*}
      [0, r_0{[} \times \bd M \to M, \quad
      (t,x) \mapsto \exp^M_x(t \mathrm n_x)
    \end{equation*}
    is a diffeomorphism onto its image where $\mathrm n_x \in T_xM$ is
    the unit normal inward vector at $x \in \bd M$. Set
    $\bd_\tau M := \{ x \in M \mid d(x,\bd M) < \tau\}$.
  \item The injectivity radius of $\bd M$ as a $(d-1)$-dimensional
    manifold is bounded from below by $r_0$.
  \item We have normal boundary coordinates at $x_0 \in \bd M$, i.e.,
    \begin{equation}
      \label{eq:norm.coord.bd}
      \phi_{x_0} \colon {[}0, r_0 {[} \times B_{\bd M}(x_0,r_0) \to M, \quad
      (t,x) \mapsto \exp^M_x(t \mathrm n_x).
    \end{equation}
  \item The injectivity radius of $M \setminus \bd_{2r_0/3}M$ is bounded
    from below by $r_0/3$. In particular, (inner) normal coordinates
    \begin{equation*}
      \phi_x \colon B(0,r_0/3) \to M, \quad v \to \exp^M_x(v)
    \end{equation*}
    exist where $x \in M \setminus \bd_{2 r_0/3}M$.
  \item
    \label{enum:bdd.curv}
    We have
    \begin{equation*}
      |(\nabla^M)^k R | \le C_k \und
      |(\nabla^{\bd M})^k \ell| \le C_k, \qquad \text{for all $k \ge 0$,}
    \end{equation*}
    where $\nabla^M$ and $\nabla^{\bd M}$ are the covariant
    derivatives in $M$ and $\bd M$, resp., $R$ the Riemann curvature
    tensor of $M$ and $\ell$ the second fundamental form of $\bd M$ in
    $M$.
  \end{enumerate}
  We refer to an atlas $\{\phi_{x_0},\phi_x\}$ of the above type~(iii)
  and~(iv) as a \emph{normal atlas}.

  Note that we can replace~(v) by the following condition
  (cf.~\cite[Prop.~3.7]{Schick-96}

  and~\cite[Thm.~2.5~(c)]{Schick-01}):
  \begin{enumerate}
  \item[(v')] Denote by $g_{ij}$ the metric components in (boundary)
    normal coordinates and by $g^{ij}$ the components of its inverse.
    We assume that there exists $C'_0>0$ such that
    \begin{equation}
      \label{eq:unif.elliptic}
      (C'_0)^{-1} |v|^2 \le \sum_{ij} g_{ij}(x) v_i v_j \le C'_0 |v|^2
    \end{equation}
    for all $x$ in the chart, $v \in \RR^d$. Furthermore,
    we assume that for each $k \in \NN, k \ge 1$ there exists a universal
    constant $C'_k>0$ such that
    \begin{equation*}
      |D^\kappa g_{ij}(x)| \le C'_k \und
      |D^\kappa g^{ij}(x)| \le C'_k
    \end{equation*}
    for all $x$, all multi-indices $\kappa$ with $|\kappa| \le k$ and
    all $k \ge 1$. Here, $D^\kappa$ denotes the partial derivative with
    respect to the coordinates.
  \end{enumerate}
\end{dfn}

We now explain how the concept of bounded geometry fits into the
framework of relatively bounded families of metrics introduced in
Definition~\ref{def:rel.bdd}:
\begin{lem}
  \label{lem:rel.bd}
  Let $(X,g_0)$ be a Riemannian covering manifold with compact
  quotient. Let $\{g_\omega\}_\omega$ be a family of Riemannian
  metrics, relatively bounded with respect to $g_0$. Then
  $(X,g_\omega)$ is of bounded geometry with constants $(r_0,C_k)$
  \emph{independent} of $\omega$.
\end{lem}
\begin{proof}
  Let us first show that $(X,g_0)$ is of bounded geometry. Since $X$
  has no boundary, we only have to verify~(iv) and~(v'): Obviously,
  the injectivity radius is bounded from below by $\rho_0>0$.
  Furthermore, if we introduce a so-called \emph{periodic atlas},
  namely a lift of a finite atlas on the compact quotient, it is
  clear, by compactness of the quotient and periodicity of the metric,
  that its components $g_{ij}$ with respect to a periodic atlas
  fulfill the estimates in~(v').

  \sloppy Now, the injectivity radius of $(X,g_\omega)$ is still
  bounded from below by $\rho_0 (C_{\rel,0})^{-1/2}$, due
  to~\eqref{quasiisom}.  The estimate~\eqref{eq:unif.elliptic} follows
  similarly.  Furthermore, the coordinate derivatives $D^\kappa
  g_{\omega,ij}$ can be expressed in terms of covariant derivatives
  $\nabla^l A_\omega$ for all $l \le |\kappa|$, since
  $g_\omega(v,v)=g_0(A_\omega v, v)$ and
  \begin{equation}
    \label{eq:part.cov}
    \partial_{i_1} \dots \partial_{i_k} =
    \nabla_{\partial_{i_1}} \dots \nabla_{\partial_{i_k}} +
      \sum_{|\kappa| < k} p_\kappa \partial^\kappa
  \end{equation}
  on tensor fields, where $p_\kappa$ is a polynomial depending only on
  the metric $g_0$ and its first $k-1$ derivatives. Here, $\nabla$ is
  the covariant derivative with respect to the periodic metric $g_0$.
  Finally, the \emph{uniform} bounded geometry of $(X,g_\omega)$
  follows from~\eqref{quasiisom} and~\eqref{gradbound}.
\end{proof}

 Let $(X,g_0)$ be a Riemannian covering manifold
with covering group $\Gamma$ and compact quotient. We fix a
(relatively compact) fundamental domain $\cF$. For any subset $I
\subset \Gamma$ let $\Lambda_0(I)$ be the $I$-agglomerate defined
in~\eqref{def:agglo}.  Furthermore, let $\Lambda(I)$ be the smoothed
version of $\Lambda_0(I)$ as constructed in~\cite[pp.~593]{Brooks-81}
and satisfying~\eqref{e-agglo-incl}.
\begin{lem}
  \label{lem:rel.bd.aggl}
  Let $\{g_\omega\}_\omega$ be a family of Riemannian metrics on $X$,
  relatively bounded with respect to $g_0$. Then
  $(\Lambda(I),g_\omega)$ and $(X \setminus \Lambda(I),g_\omega)$ are
  of bounded geometry with constants $(r_0,C_k)$ \emph{independent} of
  $\omega$ and $I \subset \Gamma$.
\end{lem}
Note that the constants $(r_0,C_k)$ of the previous lemma might differ
from the ones found in Lemma~\ref{lem:rel.bd}.
\begin{proof}
  After showing that $(\Lambda(I),g_0)$ and $(X \setminus
  \Lambda(I),g_0)$ are of bounded geometry, the general result follows
  as in the previous proof. Note that the construction of Brooks
  yields the following property of the boundaries of the smoothed
  agglomerates $\Lambda(I)$: There are finitely many relatively
  compact smooth hypersurfaces $H_1, \dots, H_n \subset X$ with
  boundaries, such that for every finite $I \subset \Gamma$ the
  boundary $\bd \Lambda(I)$ can be covered by $\Gamma$-translates of
  these finitely many hypersurfaces, i.e., for each $I$ there exists
  $N \in \NN$ and $\{\gamma_j\}_{1 \le j \le N}$ and a map $\sigma
  \colon \{1, \dots, N\} \to \{1, \dots, n\}$ such that
  \begin{equation*}
    \bd \Lambda(I) = \bigcup_{j=1}^N \gamma_j H_{\sigma(j)}.
  \end{equation*}
  Note that $n$ (the number of hypersurfaces) does \emph{not} depend
  on $I$, and that only \emph{finitely} many hypersurfaces are needed
  is due to the fact that, up to translates, the local shape of $\bd
  \Lambda(I)$ depends only on the geometry of $\cF$ and its nearest
  neighbors. This finiteness, together with the periodicity of
  $(X,g_0)$ ensures that all properties of
  Definition~\ref{def:bdd.geo} (with appropriate constants $r_0$,
  $C_k$) for $(\Lambda(I),g_0)$ and its complement are satisfied.
  Obviously, the constants $r_0$ and $C_k$ are independent of $I$.
\end{proof}


\subsection{Equivalence of Sobolev norms}

To show the equivalence of the local Sobolev norm with the others it
is important to ensure that the normal charts $\phi_{x_0}$ and
$\phi$ in Definition~\ref{def:bdd.geo} satisfy the following
conditions, which we call \emph{admissible}:
\begin{dfn}
  \label{def:admissible}
  Let $M$ be a differentiable manifold (with or without boundary).
  An atlas $\cA=\{\phi_\alpha,\chi_\alpha\}_\alpha$ of $M$ with
  coordinate charts $\phi_\alpha\colon V_\alpha \to U_\alpha \subset M$
  and subordinated functions $\chi_\alpha$ is called
  \emph{admissible} with constants $N_0 \in \NN$ and $\hat C_k>0$
  for $k\in\NN, k \ge 0$, if the following properties are satisfied:
  \renewcommand{\labelenumi}{(\roman{enumi})}
  \begin{enumerate}
  \item The partial derivatives of all coordinate change maps
    $\phi_{\alpha_1, \alpha_2} := \phi_{\alpha_1}^{-1} \circ
    \phi_{\alpha_2}$ are bounded up to all orders,
    i.e.,
    \begin{equation*}
      |\partial^\kappa \phi_{\alpha_1,\alpha_2} | \le \hat C_k
    \end{equation*}
    for all multi-indices $|\kappa| \le k$ and all $k$.
  \item The multiplicity of the covering of the altas, i.e., the
    supremum of the number of neighbors of every fixed chart $U_\alpha$
    is bounded from above by $N_0$.
  \item We have
    $\sum\nolimits_\alpha \chi_\alpha^2 = 1$ and
    \begin{equation*}
      |\partial^\kappa \chi_\alpha| \le \hat C_k
    \end{equation*}
    for all multi-indices $|\kappa| \le k$ and all $k$.
  \end{enumerate}
\end{dfn}

The proof of the following lemma can be found in~\cite[Prop.~3.8
and~3.22]{Schick-96} or~\cite[Props.~3.2, 3.3]{Schick-01}:
\begin{lem} 
  \label{lem:part.1}
  Suppose that $M$ (with or without boundary) is of bounded geometry
  with constants $C_k$ and $r_0$. Then there are constants $N_0 \in \NN$
  and $\hat C_k>0$, depending only on $C_k$ and $r_0$, such that
  an appropriate subatlas of the normal atlas in
  Definition~\ref{def:bdd.geo} is admissible with constants $N_0$
  and $\hat C_k$.
\end{lem}

We now show that the different Sobolev spaces defined in
Section~\ref{sec:sob.spaces} have equivalent norms.
\begin{lem}
  \label{lem:eq.loc.glob}
  Suppose that $(M,g)$ is of bounded geometry. Then the local Sobolev
  space $W^k(M,\cA)$ defined with respect to an admissible atlas $\cA$
  (see Definition~\ref{def:admissible}) and the global Sobolev space $W^k(M,g)$
  agree and have equivalent norms
  \begin{equation}
    \label{eq:loc.glob}
    \frac 1 {C_\sob}\| u \|_{W^k(M,g)} \le
    \|u \|_{W^k(M,\cA)} \le
    C_\sob \| u \|_{W^k(M,g)},
  \end{equation}
  where $C_\sob$ depends only on the constants of bounded geometry,
  namely $C_k$ and $r_0$.
\end{lem}
\begin{proof}
  We only sketch the proof.  For $k=0$ this follows immediately
  from~\eqref{eq:unif.elliptic}. For the higher derivatives we
  use~\eqref{eq:part.cov}, in order to express partial derivatives
  recursively by covariant derivatives, as well as the properties of
  the atlas in Definition~\ref{def:admissible}.
\end{proof}

Denote by $H=-\Delta_M^\Dir \ge 0$ the Laplace operator on $(M,g)$ with
Dirichlet boundary conditions (if $\bd M \ne \emptyset$). Comparing
the graph norm Sobolev space defined via $H$ with the global Sobolev
spaces needs elliptic estimates.
\begin{thm}
  \label{thm:ell.reg}
  Suppose that $(M,g)$ is of bounded geometry. Then for $m \ge 0$ the
  global Sobolev space $W^{2m}(M,g)$ and the graph Sobolev space
  $W^{2m}(M,-\Delta_M^\Dir)$ have equivalent norms
  \begin{multline}
    \label{eq:ell.reg}
    \frac 1 {C_\sob'}\| u \|_{W^{2m}(M,g)} \le
    \|u \|_{W^{2m}(M,H)} =
    \| (-\Delta_M^\Dir + 1)^m u \|_{L^2(M,g)} \\\le
    C_\sob' \| u \|_{W^{2m}(M,g)}
  \end{multline}
  for $u \in W^{2m}(M,-\Delta_M^\Dir)$, where $C_\sob'$ depends only on
  the constants of bounded geometry, namely $C_k$ and $r_0$.
\end{thm}
\begin{proof}
  The second inequality can easily be seen using the local Sobolev
  space, since $\Delta_M^\Dir$ contains the metric and its derivative
  and the fact that the local and global Sobolev spaces have
  equivalent norms by the last lemma.  The proof of the first
  inequality in the case $\bd M=\emptyset$ can be found e.g.\
  in~\cite[Thm.~1.3]{Dodziuk-81} (with a correction
  in~\cite[Sec.~2]{Salomonsen-01}).  The case with boundary can be
  found in~\cite[Sec~4]{Schick-96}.
\end{proof}

The next lemma is needed in the proof of Lemma~\ref{hilf1} and a
simple consequence of the product rule:
\begin{lem}
  \label{lem:mult.op} Let $k\in \NN$ be given. 
  Let $f$ be a smooth function on $(M,g)$ such that its covariant
  derivatives $\nabla^i f$ are pointwise bounded by a constant
  $C_{f,k}$ for all $i=0, \dots, k$. Then the multiplication operator
  \begin{equation*}
    M_f \colon W^k(M,g) \longrightarrow  W^k(M,g), \qquad
    \psi \mapsto f \psi
  \end{equation*}
  is bounded, and its norm is a universal constant in $k$ and
  $C_{f,k}$.
\end{lem}


\subsection{Extension operators}
\label{ssec:ext.op}

Our next result deals with an extension operator. Let $(X,g)$ be a
complete $d$-dimensional Riemannian manifold and $M \subset X$ be a
submanifold of the same dimension with smooth boundary. Suppose that
$M$ and $X \setminus M$ are of of bounded geometry with constants
$r_0$ and $C_k$. Then $X$ is also of bounded geometry with the same
constants.

Let $\cA$ and $\cA'$ be normal atlasses of $M$ and $X \setminus M$.
An associated atlas $\wt \cA$ of $X$ is given by the inner normal
charts of $\cA$ and $\cA'$ and by extensions of the normal boundary
charts $\phi_{x_0} \colon {[}0,r_0{[} \times B_{\bd M}(x_0,r_0) \to M$
of $\cA$ to collar maps $\wt \phi_{x_0} \colon {]}-r_0,r_0{[} \times
B_{\bd M}(x_0,r_0) \to X$. Clearly, by Lemma~\ref{lem:part.1}, we can
choose an admissible subatlas of $\wt A$ (and the corresponding
subatlas of $\cA$). We denote the subatlasses by the same symbols
$\cA$ and $\wt \cA$.

We denote inner and boundary charts on $M$ by $\phi_\alpha \colon
V_\alpha \to U_\alpha \subset M$ and on $X$ by $\widetilde \phi_\alpha
\colon \widetilde V_\alpha \to \widetilde U_\alpha \subset X$ and
similarly, we denote by $\chi_\alpha$ and $\wt \chi_\alpha$ the
associated partitions of unity.  Note that now, $V_\alpha$ is an open
subset of the half-space $\RR^d_+ = [0,\infty{[} \times \RR^{d-1}$
whereas $\wt V_\alpha$ is open in $\RR^d$.  Clearly, we can assume
that also the chart domains satisfy $\wt V_\alpha \cap \RR^d_+ =
V_\alpha$.

Now we can define a Sobolev extension operator by a local procedure,
using the extension operator $\mathcal E_0 \colon W^k(\RR^d_+) \to
W^k(\RR^d)$ on the half-space $\RR^d_+$ (see
e.g.~\cite[Thm.~7.25]{GilbargT-77}).

\begin{thm}
  \label{thm:ext.op}
  Suppose that $M$ and $X \setminus M$ have bounded geometry with
  constants $r_0$ and $C_l$. Suppose, in addition, that $\cA$ and $\wt
  \cA$ are atlasses of $M$ and $X$ as defined above. Then, for every
  $k \in \NN$, there exists an extention operator
  \begin{equation}
    \label{eq:ext.op}
    \mathcal E \colon W^k(M,\cA) \to W^k(X,\wt \cA),
  \end{equation}
  such that $\|\mathcal E\|$ only depends on $k$, $C_l$ and $r_0$.
\end{thm}
\begin{proof}
  We set
  \begin{equation}
    \label{eq:def.ext.op}
    \mathcal E u (x) :=
    \sum_{\alpha \in \cA} \widetilde \chi_\alpha(x) \cdot
       \bigl(
          \mathcal E_0 (\chi_\alpha u_\alpha) \bigr)
       (\widetilde \phi_\alpha^{-1} x)
  \end{equation}
  for $x \in X$ and $u \in W^k(M,\cA)$.  Note
  that~\eqref{eq:def.ext.op} is well-defined: In a neighborhood of
  $x$ at most $N_0$ terms are non-zero, so the sum is essentially
  finite.  In addition, $\chi_\alpha u_\alpha \in W^k(\RR^d_+)$ and
  $(\widetilde \chi_\alpha \circ \widetilde \phi_\alpha^{-1}) \cdot
  \mathcal E_0 (\chi_\alpha u_\alpha) \in W^k(\widetilde V_\alpha)$
  with compact support in $\widetilde V_\alpha$.  Finally, $\mathcal E
  u \in W^k(X)$. Clearly, \eqref{eq:def.ext.op} defines an extension
  operator.

  For the norm estimate, we have
  \begin{align*}
    \|\mathcal E u \|^2_{W^k(X,\wt \cA)} &=
    \sum_{\alpha' \in \widetilde \cA}
      \|\widetilde \chi_{\alpha'} (\mathcal E u)_{\alpha'}
           \|^2_{W^k(\widetilde V_{\alpha'})}\\&=
    \sum_{{\alpha'} \in \widetilde \cA}
      \Bigl\| \widetilde \chi_{\alpha'}
         \sum_{\alpha \in \cA}
          \bigl[ \chi_{\alpha} \cdot  \mathcal E_0
                          (\chi_{\alpha} u_{\alpha}) \bigr]
          \circ (\widetilde\phi_{\alpha}^{-1}
              \circ \widetilde \phi_{\alpha'})
      \Bigr\|^2_{W^k(\widetilde V_{\alpha'})}\\& \le
    \sum_{{\alpha'} , \alpha} N_0
      \Bigl\|
        \widetilde \chi_{\alpha'}
        \bigl[ \chi_{\alpha} \cdot  \mathcal E_0
                      (\chi_{\alpha} u_{\alpha}) \bigl]
                 \circ (\widetilde \phi_{\alpha}^{-1} \circ
                       \widetilde \phi_{\alpha'})
       \Bigr\|^2_{W^k(V_{\alpha'})}
  \end{align*}
  where the last sum is taken over all $\alpha \in \cA$, $\alpha' \in
  \widetilde \cA$ such that $\widetilde U_\alpha \cap \widetilde
  U_{\alpha'} \ne \emptyset$. Due to Lemma~\ref{lem:part.1}, there are
  at most $N_0$ indices $\alpha'$ for a fixed $\alpha$, and we can
  estimate the remaining sum (using the product and chain rule) by a
  constant, depending only on $\hat C_k$ and $k$, multiplied with
  $N_0^2 \sum_\alpha \| \mathcal E_0(\chi_\alpha u_\alpha)
  \|^2_{W^k(M,\cA)} \le N_0^2 \| \mathcal E_0 \|^2 \, \|u
  \|^2_{W^k(M,\cA)}$.
\end{proof}


\subsection{Sobolev embedding}

In this  subsection, we show that there is a continous embedding
of the graph Sobolev space defined with respect to the Laplacian
$H:=-\Delta_X \ge 0$ into $\Cb X$, where $\Cb X$ denotes the space of
\emph{bounded} continuous functions on $X$.
\begin{thm}
  \label{thm:sob.emb}
  Suppose that $(X,g)$ is a complete $d$-dimensional manifold with
  sectional curvature $K$ bounded by $|K(x)| \le K_0$ and positive
  injectivity radius $r_0 := \injrad X>0$. Suppose, in addition, that
  $m \ge d/4 + 1$. Then the embedding
  \begin{equation}
    \label{eq:sob.emb}
    H^{2m}(X,-\Delta_X) \to \Cb X
  \end{equation}
  is defined and bounded with norm depending only on $m$, $d =\dim X$,
  $r_0$ and $K_0$.
\end{thm}
\begin{proof}
  The theorem follows directly from
  \begin{multline*}
    |\psi(x)| \le
    c(d) \sum_{i=0}^m r^{-d/2+i}
         \| \Delta_X^i \psi \|_{L^2(B(x,r),g)}\\ \le
    c'(d,m) r^{-d/2} \| (-\Delta_X+1)^m \psi\|_{L^2(X,g)},
  \end{multline*}
  for $m \ge d/4 + 1$, $x \in X$ and $r \le \min\{K_0^{-1/2}, r_0,1\}$
  (cf.~\cite[Prop.~1.3]{CheegerGT-82}), and the spectral calculus.
  Consequently, we obtain
  \begin{equation*}
    \| \psi \|_\infty \le
    c'(d,m) \max \bigl\{K_0^{d/4},r_0^{-d/2},1 \bigr\} \,
                         \| \psi\|_{W^{2m}(X,-\Delta_X)}.
  \end{equation*}
\end{proof}


\subsection{A Hilbert-Schmidt norm estimate}

The following standard estimate is used in Section \ref{s-TraceRes}.
Since we could not find a reference, we present it with a proof, for
the reader's convenience.

Let $(X,m)$ be an measurable space such that $L^2(X,m)$ is separable.
In addition, let $Y$ be a topological space with finite Borel measure
$m'$. Denote by $\Cb Y$ the space of bounded continuous functions on
$Y$ with supremum norm $\Vert \cdot \Vert_\infty$. We denote $J \colon
\Cb Y \to L^2(Y,m')$ the canonical embedding.
\begin{thm}
  \label{t-HSkernel}
 Assume that
  \begin{equation*}
    K \colon  L^2(X,m) \to \Cb Y
  \end{equation*}
  is a bounded operator.  Then the composition $J K \colon L^2(X,m)
  \to L^2(Y,m')$ is a Hilbert-Schmidt operator with Hilbert-Schmidt
  norm bounded by
  \begin{equation*}
    \Vert J K \Vert_{\cJ_2} \le \Vert K \Vert (m'(Y))^{1/2}.
  \end{equation*}
\end{thm}

\begin{proof}
  Let $\{\varphi_n\}_n$ be an orthonormal base of $L^2(X,m)$. Since
  for fixed $y \in Y$, the map $L^2(Y,m') \to \CC$, $f \mapsto Kf(y)$
  is a bounded functional, there exists $g_y \in L^2(X,m)$ such that
  $Kf(y) = \langle g_y, f \rangle$ and $\| g_y \| \le \Vert K \Vert$.
  Denoting $c_n(y) := \la g_y, \varphi_n \ra$ the Fourier coefficients
  of $g_y$, we obtain
  \begin{equation*}
    \sum_n |c_n(y)|^2 
    = \Vert g_y \Vert_{L_2}^2 
    = | K g_y(y) | 
    \le \Vert K \Vert \Vert g_y \Vert_{L_2},
  \end{equation*}
  and conclude that $\sum_n |c_n(y)|^2 \le \Vert K \Vert^2$. Moreover,
  we have
  \begin{equation*}
    Kf(y) 
    = \la g_y,f \ra 
    = \sum_n c_n(y) \la \varphi_n, f \ra.
  \end{equation*}
  The function
  \begin{equation*}
    k_N \colon X \times Y \to \mathbb C, \qquad 
    k_N(x,y) := \sum_{n=1}^N c_n(y) \varphi_n(x)
  \end{equation*}
  is obviously measurable and in $L^2(X \times Y)$. Its $L^2$-norm can
  be estimated uniformly as
  \begin{equation*}
    \| k_N \|^2_{L^2(X \times Y)}
    = \sum_{n=1}^N \int_Y |c_n(y)|^2 \dd m'(y) \le \Vert K \Vert^2 m'(Y).
  \end{equation*}
  Similarly, it can be shown that $\{k_N\}_N$ is actually a Cauchy
  sequence in $L^2(X \times Y)$ with limit $k$. Denote the operators
  associated to $k_N$ and $k$ by $K_N$ and $\wt K$, respectively.  It
  remains to show that $\wt K = J K$. For $f
  \in L^2(X,m)$ we have 
  \begin{multline*}
    \Vert \tilde K f - K_N f \Vert^2
    = \int_Y \Big| \int_X (k(x,y) - k_N(x,y)) f(x)
              \dd m(x) \Big|^2 \dd m'(y) \\
   \le \int_Y \int_X |k(x,y) - k_N(x,y)|^2 \dd m(x) \Vert f
      \Vert_{L^2}^2 \dd m'(y)\\
    \le \Vert k- k_N \Vert_{L^2(X \times Y)}^2 \Vert f \Vert_{L^2}^2.
  \end{multline*}
  Passing to a subsequence we conclude that $\lim_{N \to \infty} K_N
  f(y) \to \tilde K f(y)$ for almost all $y \in Y$. On the other hand,
  we have
  \begin{equation*}
     K_N f(y) 
     = \Bigl\langle \sum_{n=1}^N c_n(y) \varphi_n,f \Bigr\rangle 
               \to \la g_y,f \ra
     = (K f) (y)
  \end{equation*}
  for all $y \in Y$ and hence, $\tilde K f(y) = K f(y)$ for almost all
  $y \in Y$. Since $Kf$ is continuous and bounded, we conclude $\wt K
  f = Kf = JKf$.
\end{proof}

\section*{ Acknowledgements}
This work was
supported in part by the DFG through the Sonderforschungsbereich~237,
the Schwerpunktprogramm~1033, and grants no.~Ve~253/1-1 and
Ve~253/2-1.


\newcommand{\etalchar}[1]{$^{#1}$}
\def\cprime{$'$} \def\cprime{$'$}


\end{document}